\documentclass{article}
\usepackage{amsthm,amsmath,amssymb,bigstrut,mathrsfs,mathtools,mathabx,url}
\usepackage{combgames}

\newcommand\A{\ensuremath{\mathcal{A}}}
\newcommand\D{\ensuremath{\mathcal{D}}}
\newcommand\E{\ensuremath{\mathcal{E}}}
\newcommand\Mhat{\ensuremath{\hat{\mathcal{M}}}}
\renewcommand\L{\mathcal{L}}
\renewcommand\P{\ensuremath{\mathscr{P}}}
\newcommand\N{\mathscr{N}}

\newcommand\sL{\mathscr{L}}
\newcommand\U{\ensuremath{\mathcal{U}}}
\newcommand\V{\ensuremath{\mathcal{V}}}
\newcommand\cl{\mathrm{cl}}
\newcommand\up{\mathrm{up}}
\newcommand\birthday{\mathrm{b}}
\newcommand\fbirthday{\tilde{\mathrm{b}}}
\newcommand{\sh}{\textrm{\raisebox{0.85pt}{\tiny \#}}}

\newcommand\cgtwo[2]{\combgame{\{#1 \cgslash #2\}}}
\newcommand\fatGL{\mathspoofwidth{G^R}{G^L}}
\newcommand\oast{{\scriptscriptstyle\mathrlap{\otimes}\oplus}}
\newcommand\boldsc[1]{{\fontencoding{T1}\scshape\selectfont #1}}
\newcommand\letree[1]{\cgtree[unit=0.25cm,nodesep=0]{#1}} 
\newcommand\en{\bullet} 

\newtheorem{theorem}{Theorem}[section]
\newtheorem{lemma}[theorem]{Lemma}
\newtheorem{proposition}[theorem]{Proposition}
\newtheorem{definition}[theorem]{Definition}
\newtheorem{corollary}[theorem]{Corollary}
\newtheorem*{openproblem}{Open Problem}
\newenvironment{example}{\begin{proof}[Example]}{\end{proof}}
\newenvironment{examples}{\begin{proof}[Examples]}{\end{proof}}

\newlength\spoofwidthtmp
\newcommand\spoofwidth[3][c]{%
\settowidth{\spoofwidthtmp}{#2}%
\makebox[\spoofwidthtmp][#1]{#3}%
}
\newcommand\mathspoofwidth[3][c]{\spoofwidth[#1]{$#2$}{$#3$}}

\title{On the General Dead-Ending Universe of Partizan Games\\\vspace{0.1in}{\large \textsc{preprint}}}
\author{Aaron N. Siegel}

\begin{document}

\maketitle

\begin{abstract}
The universe $\E$ of dead-ending partizan games has emerged as an important structure in the study of mis\`ere play. Here we attempt a systematic investigation of the structure of $\E$ and its subuniverses. We begin by showing that the dead-ends exhibit a rich ``absolute'' structure, in the sense that they behave identically in any universe in which they appear.

We will use this result to construct an uncountable family of dead-ending universes and show that they collectively admit an uncountable family of distinct comparison relations. We will then show that whenever the ends of a universe $\U \subset \E$ are computable, then there is a constructive test for comparison modulo~$\U$.

Finally, we propose a new type of generalized simplest form that works for arbitrary universes (including universes that are not dead-ending), and that is computable whenever comparison modulo $\U$ is computable. In particular, this gives a complete constructive theory for subuniverses of $\E$ with computable ends. This theory has been implemented in \texttt{cgsuite} as a proof of concept.

As an application of these results, we will characterize the universe generated by mis\`ere \textsc{Domineering}, and we will compute the mis\`ere simplest forms of $2 \times n$ \textsc{Domineering} rectangles for small values of~$n$.
\end{abstract}

\section{Introduction}

The study of mis\`ere play of partizan games, long considered mysterious, has exploded into an area of active interest over the past 15 years. Much of this work has focused on localizing the usual equivalences to a restricted subset of partizan games. Denote by $\mathcal{M}$ the universe of all partizan games; then for any subset $\A \subset \mathcal{M}$ that is closed under addition, one can define the usual relations:
\begin{align*}
G \geq_\A H &\quad\text{if}\quad o(G + X) \geq o(H + X) \quad\text{for all}\quad X \in \A \\
G \equiv_\A H &\quad\text{if}\quad o(G + X) = o(H + X) \quad\text{for all}\quad X \in \A
\end{align*}
Here and throughout this paper, $o(G)$ denotes the mis\`ere outcome class of~$G$; see \cite{Sie13} for the necessary background. It is customary to define these relations only for $G,H \in \A$, but in this paper we will have occasion to consider them for $G$ and $H$ ranging over all partizan games, and so define them as such.

A well-known example is the universe $\D$ of \textbf{dicotic games} (or simply~\textbf{dicots}): a game $G \in \mathcal{M}$ is a dicot if every nonempty subposition of $G$ has at least one Left option \text{and} at least one Right option. Complementary to the dicots are the \textbf{ends}:

\begin{definition}
A partizan game $G$ is a \textbf{Left (resp.~Right) end} if $G$ has no Left (resp.~Right) option.
\end{definition}
Note that $0$ is the unique game that is both a dicot and an end.

The behavior of dicots in mis\`ere play is particularly clean, since neither player can win until an empty subposition is reached; conversely, ends are the source of a great deal of combinatorial complexity, and their general structure is still not fully understood. Recent work due to Milley, Renault, and others \cite{LMNRS} \cite{Mil13} \cite{MR13} has identified a natural subclass of ends that has a more approachable structure.

\begin{definition}
A partizan game $G$ is a \textbf{Left (resp.~Right) dead-end} if every subposition of $G$ is a Left (resp.~Right) end. The set of Left dead-ends is denoted by~$\mathcal{L}$.
\end{definition}

\begin{definition}
A partizan game $G$ is a \textbf{dead-ending game} if every end in $G$ is a (Left or Right) dead-end. The set of dead-ending games is denoted by~\E.
\end{definition}

Related work of Larsson, Nowakowski, and Santos \cite{LNS} has revealed a rich hierarchy of \textbf{(absolute) universes} of partizan games.

\begin{definition}[Larsson--Nowakowski--Santos]
\label{definition:universe}
A set $\U \subset \mathcal{M}$ is a \textbf{universe}\footnote{Some authors use the term \textbf{universe} more generally, requiring only that conditions (i)-(iii) be satisfied, with an \textbf{absolute universe} satisfying additionally~(iv). We prefer to reserve the term \textbf{universe} for sets that satisfy (i)-(iv) and will often omit the adjective ``absolute.''} (or \textbf{universally closed}) if it satisfies each of four closure properties:
\begin{enumerate}
\item[(i)] (hereditary closure) If $G \in \U$, then every $G^L \in \U$ and every $G^R \in \U$.
\item[(ii)] (additive closure) If $G \in \U$ and $H \in \U$, then $G + H \in \U$.
\item[(iii)] (conjugate closure) If $G \in \U$, then $\overline{G} \in \U$.
\item[(iv)] (dicotic closure) If $\mathscr{G}^L \subset \U$ and $\mathscr{G}^R \subset \U$ are nonempty finite subsets of~\U, then $\combgame{\{\smash{\mathscr{G}^L}|\smash{\mathscr{G}^R}\}} \in \U$.
\end{enumerate}
We say a set $\A \subset \mathcal{M}$ is \textbf{simply closed} (or just \textbf{closed}) if it satisfies (i)--(ii), but not necessarily (iii)--(iv).
\end{definition}
Every universe $\U$ satisfies $\D \subset \U \subset \mathcal{M}$; we say $\U$ is a \textbf{dead-ending universe} if also $\U \subset \E$.

It is natural to ask three fundamental questions of each universe~\U:
\begin{enumerate}
\item[(Q1)] Is there a constructive test for $G \geq_\U H$, i.e., is $\geq_\U$ computable?
\item[(Q2)] Do games in $\U$ admit simplest forms, and if so, how does one calculate them?
\item[(Q3)] How many distinct games (up to equivalence modulo~$\U$) are there of birthday ${\leq n}$?
\end{enumerate}

After a survey of preliminaries in Section~\ref{section:preliminaries}, the bulk of this paper consists of two main developments:
\begin{itemize}
\item Sections \ref{section:left-ends} and \ref{section:comparison} attempt a systematic study of dead-ending games. In Section~\ref{section:left-ends}, we show that the Left dead-ends have a surprisingly rigid (``absolute'' or ``multiversal'') structure: they behave the same (relative to one another) in any universe in which they appear. In particular, \emph{comparison} and \emph{simplest forms} for Left dead-ends are independent of the universe under consideration.

In Section~\ref{section:comparison}, we apply these results to show that there are uncountably many subuniverses $\U \subset \E$, which collectively admit uncountably many distinct comparison relations~$\geq_\U$. Moreover, for any such~\U, the comparison relation $\geq_\U$ is completely determined (over all of~$\mathcal{M}$) by its behavior on~\D. This shows that not all comparison relations are computable, even when restricted to~\D.

We will show, additionally, that $\geq_\U$ \emph{is} in fact computable whenever the Left ends of $\U$ are computable (in a sense defined below, cf.~Corollary~\ref{corollary:computable}). In particular, $\geq_\U$~is computable for any dead-ending $\U$ that is finitely generated as an extension of~\D.

\item Section \ref{section:invariant-forms} addresses the question of simplest forms in the general universe~\U. We will show that there are mis\`ere universes $\U$ that do not admit simplest forms for all $G \in \U$, in the traditional sense of a simplest $K \in \U$ with $K \equiv_\U G$. However, we will also show that simplest forms can nonetheless be obtained by enlarging $\U$ to a wider class $\Mhat$ of \textbf{augmented games}: for every universe $\U$ and every $G \in \U$, there is a unique simplest $K \in \Mhat$ with $K \equiv_\U G$. This construction works in any universe; it does not require the assumption $\U \subset \E$.

Moreover, these augmented simplest forms can be effectively computed whenever the comparison relation $\geq_\U$ is computable. This gives, in combination with the preceding results, computable simplest forms for every computable dead-ending universe $\U \subset \E$, including for any such $\U$ that is finitely generated as an extension of~\D. General comparison and simplest forms for computable dead-ending universes have been implemented in \texttt{cgsuite}~\cite{cgsuite} as a proof of concept.
\end{itemize}

As an application of these results, in Section~\ref{section:domineering} we will give a complete description of the universe of mis\`ere \textsc{Domineering}, including an exact analysis of all $1 \times n$ \textsc{Domineering} rectangles, as well as the simplest forms of $2 \times n$ rectangles for small~$n$ (calculated modulo the \textsc{Domineering} universe).

Finally, Section~\ref{section:whats-next} discusses some intriguing further directions, including generalizations to scoring universes and to non-dead-ending mis\`ere universes.


\section{Background and Preliminaries}
\label{section:preliminaries}

Partizan games in mis\`ere play remained largely inscrutable until 2007, when Mesdal and Ottaway obtained the first significant results in the theory~\cite{MO07}. Later that same year, the present author built on their work to answer (Q1)--(Q3) in the case of the full mis\`ere universe~$\mathcal{M}$ \cite{Sie13}~\cite{Sie15pmcf}. In particular,


\begin{theorem}[Siegel \cite{Sie13} \cite{Sie15pmcf}]
\label{theorem:full-comparison}
$G \geq_\mathcal{M} H$ if and only if:
\begin{enumerate}
\item[(a)] For every $G^R$, either there is some $H^R$ with $G^R \geq_\mathcal{M} H^R$, or else there is some $G^{RL}$ with $G^{RL} \geq_\mathcal{M} H$;
\item[(b)] For every $H^L$, either there is some $G^L$ with $G^L \geq_\mathcal{M} H^L$, or else there is some $H^{LR}$ with $G \geq_\mathcal{M} H^{LR}$;
\item[(c)] If $H$ is a Left end, then so is $G$;
\item[(d)] If $G$ is a Right end, then so is $H$.
\end{enumerate}
\end{theorem}

The phrasing of Theorem~\ref{theorem:full-comparison} is slightly different in~\cite{Sie13}~\cite{Sie15pmcf}; here 
we use a modified version of the more concise formulation due to Larsson et al~\cite{LNS}. The conditions (c) and (d) are commonly called the \textbf{proviso}, a term that goes back to Conway~\cite{Con01}. Following Larsson et al, we call (a) and (b) the \textbf{maintenance property}.

Theorem~\ref{theorem:full-comparison} was followed several years later by a similar analysis of dicotic games by Dorbec et al~\cite{DRSS15}.

\begin{theorem}[Dorbec et al \cite{DRSS15}]
\label{theorem:dicot-comparison}
$G \geq_\D H$ if and only if:
\begin{enumerate}
\item[(a)] For every $G^R$, either there is some $H^R$ with $G^R \geq_\D H^R$, or else there is some $G^{RL}$ with $G^{RL} \geq_\D H$;
\item[(b)] For every $H^L$, either there is some $G^L$ with $G^L \geq_\D H^L$, or else there is some $H^{LR}$ with $G \geq_\D H^{LR}$;
\item[(c)] If $H$ is a Left end, then $o(G) \geq \N$;
\item[(d)] If $G$ is a Right end, then $o(H) \leq \N$.
\end{enumerate}
\end{theorem}

Here again, we use the formulation due to Larsson et al (cf.~\cite{LMNRS}), which makes the pattern evident: the maintenance property is the same as in Theorem~\ref{theorem:full-comparison}; only the proviso differs. A major advance followed in 2016, when Larsson, Nowakowski, and Santos unveiled~\cite{LNS} a sweeping generalization of Theorems \ref{theorem:full-comparison}~and~\ref{theorem:dicot-comparison}. It is convenient to make the following definition:

\begin{definition}
Let $\U$ be a mis\`ere universe. We say that:
\begin{enumerate}
\item[(a)] $G$ is \textbf{Left \U-strong} if $o(G + X) \geq \N$ for every Left end $X \in \U$.
\item[(b)] $G$ is \textbf{Right \U-strong} if $o(G + X) \leq \N$ for every Right end $X \in \U$.
\end{enumerate}
\end{definition}

Then the general Theorem can be stated as follows:

\begin{theorem}[Larsson et al \cite{LNS}]
\label{theorem:lnsequiv}
Let $\U$ be any mis\`ere universe. Then $G \geq_\U H$ if and only if:
\begin{enumerate}
\item[(a)] For every $G^R$, either there is some $H^R$ with $G^R \geq_\U H^R$, or else there is some $G^{RL}$ with $G^{RL} \geq_\U H$;
\item[(b)] For every $H^L$, either there is some $G^L$ with $G^L \geq_\U H^L$, or else there is some $H^{LR}$ with $G \geq_\U H^{LR}$;
\item[(c)] If $H$ is a Left end, then $G$ is Left \U-strong;
\item[(d)] If $G$ is a Right end, then $H$ is Right \U-strong.
\end{enumerate}
\end{theorem}

(The full Theorem proved in \cite{LNS} is still more general than this; it also applies to universes with \textbf{atoms}, such as scoring universes. The version given in Theorem~\ref{theorem:lnsequiv} is all we will need for the bulk of this paper, though we will return to the question of scoring universes in Section~\ref{section:whats-next}.)


The maintenance property (a)--(b) is identical in every mis\`ere universe (and even in normal play), whereas the proviso (c)--(d) depends on the localized meaning of \U-strong. The following facts are readily apparent.
\begin{itemize}
\item Every Left end is automatically Left $\U$-strong, in any universe.
\item In the full universe $\U = \mathcal{M}$, the Left ends are the \emph{only} games that are Left \U-strong (this follows from the proof of Theorem~\ref{theorem:full-comparison}; cf.~\cite{Sie13}).
\item In the dicot universe $\D$, the only end is~$0$, so that $G$ is Left \U-strong if and only if $o(G) \geq \N$ (i.e., Left can win $G$ playing first).
\item If $\U \subset \V$ and $G$ is Left \V-strong, then $G$ is also Left \U-strong.
\end{itemize}

Theorem~\ref{theorem:lnsequiv} does not, in itself, give a constructive test for inequality, since the quantifier in the definition of \U-strong ranges over the entire (infinite) universe~\U. But it \emph{does} show that such a test depends only on identifying whether a given game $G$ is \U-strong.




\subsection*{Universal Closures}

If $\mathcal{A}$ is any set of games, we write $\D(\mathcal{A})$ for the \textbf{universal closure} of~$\mathcal{A}$, the smallest universe containing $\mathcal{A}$ as a subset (formally, $\D(\mathcal{A})$ can be defined as the intersection of all universes containing~$\mathcal{A}$). Similarly, write $\cl(\A)$ for the \textbf{(simple) closure} of~$\mathcal{A}$, the smallest simply closed set containing $\A$ as a subset.

Further examples of dead-ending universes include:
\begin{itemize}
\item The universe $\D(\bar{1})$, where $\bar{1}$ is shorthand for $\combgame{\{\cdot|0\}}$. Since $\bar{1}$ is a follower of every Left dead-end, $\D(\bar{1})$ is the minimal nondicotic subuniverse of~$\E$: if $\U \subset \E$ with $\U \neq \D$, then $\D(\bar{1}) \subset \U$. An element $G \in \D(\bar{1})$ can be viewed as a dicotic tree whose leaf nodes are integers.
\item The universe $\D(\bar{1}0)$, where $\bar{1}0$ is shorthand for $\combgame{\{\cdot|\bar1, 0\}}$. In Section~\ref{section:domineering} we will study this universe and show that it is, in fact, the universal closure of \textsc{Domineering}.
\end{itemize}





The following easy propositions help bring structure to the ``multiverse''.

\begin{proposition}
\label{proposition:endsgenerate}
Let $\U \subset \E$ be any dead-ending universe. Then
\[
\mathcal{D}(\U \cap \mathcal{L}) = \U.
\]
\end{proposition}

\begin{proposition}
\label{proposition:nonewends}
Let $\mathcal{A}$ be any closed set of Left dead-ends. Then
\[
\mathcal{D}(\mathcal{A}) \cap \mathcal{L} = \mathcal{A}.
\]
\end{proposition}

Proposition~\ref{proposition:endsgenerate} shows that every dead-ending universe is the closure of its Left ends; Proposition~\ref{proposition:nonewends} shows that if $\mathcal{A}$ is closed, then passing to the universal closure adds no new Left ends. Taken together, they establish a one-to-one correspondence between closed sets of Left ends and absolute universes, given by $\mathcal{A} \mapsto \D(\mathcal{A})$.

\begin{proof}[Proof of Proposition \ref{proposition:endsgenerate}]
Certainly $\mathcal{D}(\U \cap \mathcal{L}) \subset \U$. For the converse, consider any $G \in \U$ and let $\A$ be the set of ends appearing as followers of~$G$. Then $G$ is in the dicotic closure of~\A. It follows that $\U$ is generated by the set of its ends, and since it is also conjugate closed, the Proposition follows.
\end{proof}

\begin{proof}[Proof of Proposition \ref{proposition:nonewends}]
We will show that:
\begin{enumerate}
\item[(i)] If $\A$ is a closed set of Left ends, then the additive closure $\cl(\A \cup \bar{\A})$ contains no Left ends outside of~\A.
\item[(ii)] If $\mathcal{C}$ satisfies (i)--(iii) in Definition~\ref{definition:universe}, then its dicotic closure is a universe.
\end{enumerate}
Writing $\mathcal{C} = \cl(\A \cup \bar{\A})$, it follows from (i) that $\mathcal{C}$ contains no Left ends outside of~\A, and from (ii) that the dicotic closure of $\mathcal{C}$ is $\D(\A)$. Since dicotic closure also adds no new Left ends, this suffices to prove the Proposition.

\vspace{0.1in}\noindent
(i) $G + H$ is a Left end if and only if $G$ and $H$ are both Left ends. But $\A$ was taken to be closed. By induction, we may assume $G,H \in \A$, whereupon $G + H \in \A$ as well.

\vspace{0.1in}\noindent
(ii) We must show that the dicotic closure of $\mathcal{C}$ also satisfies (i)--(iii) in Definition~\ref{definition:universe}. Hereditary closure and conjugate closure are immediate. For additive closure, let $G,H \in \D(\mathcal{C})$. If $G$ and $H$ are both ends, then the conclusion follows from closure of~$\mathcal{C}$. Otherwise, we can assume by induction that all $G^L + H$, $G + H^L$, $G^R + H$, and $G + H^R$ are elements of $\D(\mathcal{C})$, and the conclusion follows from the definition of dicotic closure.
\end{proof}

\begin{definition}
If $\U = \D(\A)$, then we say that $\U$ is \textbf{generated by}~\A. We say that $\U$ is \textbf{finitely generated} if $\U = \D(\A)$ for some finite set~\A.
\end{definition}

\section{The Structure of Left Dead-Ends}
\label{section:left-ends}
\suppressfloats[t]

The starting point for any investigation of dead-ending universes is a study of the dead-ends themselves. We show here that they have an exceptionally rigid structure.

If $G$ and $H$ are Left dead-ends, then they are both Left $\U$-strong (for any~$\U$), so condition (c) in Theorem~\ref{theorem:lnsequiv} is satisfied automatically. Moreover, no Left end other than $0$ can be Right \U-strong, so the condition ``$H$~is Right \U-strong'' in (d) reduces to ``$H \cong 0$''. Since $G$ and $H$ have no Left options and no reversible moves, the entirety of Theorem~\ref{theorem:lnsequiv} reduces to the following simple criterion.

\begin{theorem}
\label{theorem:de-comparison}
Let $\U$ be a universe, and let $G$ and $H$ be Left dead-ends. Then $G \geq_\U H$ if and only if:
\begin{enumerate}
\item[(a)] For every $G^R$, there is some $H^R$ with $G^R \geq_\U H^R$; and
\item[(b)] If $G \cong 0$, then also $H \cong 0$.
\end{enumerate}
\end{theorem}


Notably, the condition (b) does not actually depend on~$\U$, so in fact Left dead-ends obey the same order-relations in any universe in which they appear: if $G \geq_\U H$ in \emph{any} universe~\U, then the same relation holds in \emph{every}~\U. We will say that comparison of Left dead-ends is an \textbf{absolute} relation, by which we mean that it is independent of~$\U$ (echoing terminology introduced by Larsson et~al). In relations involving Left dead-ends, we can therefore drop the subscript $\U$ and write simply $G \geq H$ or $G = H$ (\textbf{absolute equality}).


\begin{corollary}
For any Left dead-end $G \not\cong 0$, we have $G \not\geq 0$ and $0 \not\geq G$.
\end{corollary}

\begin{proof}
Since $G \not\cong 0$, it must have at least one~$G^R$. Therefore $G \not\geq 0$ follows from Theorem~\ref{theorem:de-comparison}(a), and $0 \not\geq G$ follows from~(b).
\end{proof}

There is another simple, but important, consequence of Theorem~\ref{theorem:de-comparison}. Denote by $\fbirthday(G)$ the \textbf{formal birthday} (or \textbf{formal rank}) of~$G$, i.e., the height of the game tree for~$G$. Similarly, write $\birthday(G)$ for the \textbf{birthday} (or \textbf{rank}) of~$G$, defined as
\[
\birthday(G) = \min\{\fbirthday(H) : H = G\}.
\]
For general games $G$, the birthday $\birthday(G)$ depends on the universe under consideration, but for dead-ends it is absolute. 

\begin{proposition}
\label{prop:de-fbirthday}
Let $G$ and $H$ be Left dead-ends, and suppose that $G \geq H$. Then ${\fbirthday(G) \leq \fbirthday(H)}$.
\end{proposition}

\begin{proof}
If $G \cong 0$, then also $H \cong 0$, so $\fbirthday(G) = \fbirthday(H) = 0$. Otherwise, let $G^R$ be any option of maximal formal rank, $\fbirthday(G^R) = \fbirthday(G) - 1$. Then there must be some $H^R$ with $G^R \geq H^R$. By induction, $\fbirthday(H^R) \geq \fbirthday(G) - 1$, whence $\fbirthday(H) \geq \fbirthday(G)$.
\end{proof}

\begin{proposition}
\label{prop:de-birthday}
Let $G$ and $H$ be Left dead-ends, and suppose that $G \geq H$. Then ${\birthday(G) \leq \birthday(H)}$.
\end{proposition}

\begin{proof}
Follows immediately from Proposition~\ref{prop:de-fbirthday}.
\end{proof}

\subsection*{Absolute Simplest Forms}

Theorem~\ref{theorem:de-comparison} leads immediately to an easy simplest form theorem for dead ends, and since equality for dead ends is an absolute relation, their simplest forms are absolute invariants.

\begin{lemma}
\label{lemma:de-subset}
Let $G$ and $H$ be Left dead-ends, and suppose that the Right options of $G$ are a nonempty subset of those of~$H$. Then $G \geq H$.
\end{lemma}

\begin{proof}
Condition (a) in Theorem~\ref{theorem:de-comparison} follows immediately from the hypotheses, and condition (b) is vacuous, since $G \not\cong 0$.
\end{proof}

\begin{lemma}
Let $G$ be a Left dead-end, and suppose that $G'$ is obtained from $G$ by eliminating a dominated option $G^{R_1}$ from~$G$. Then $G = G'$.
\end{lemma}

\begin{proof}
$G \geq G'$ by Lemma~\ref{lemma:de-subset}. To see that $G' \geq G$, we verify condition (a) in Theorem~\ref{theorem:de-comparison}: $G^{R_1}$ is dominated by some $G^{R_2}$, and every other option of $G$ is paired in~$G'$. Condition (b) is vacuous, as in Lemma~\ref{lemma:de-subset}.
\end{proof}

\begin{theorem}
\label{theorem:deadendcanonical}
Suppose that $G$ and $H$ are Left dead-ends, and assume no subposition of $G$ nor $H$ has any dominated options. If $G = H$, then in fact $G \cong H$.
\end{theorem}

The proof is essentially identical to many others in the literature, but for completeness, we give it here.

\begin{proof}
Fix any $G^R$. Then since $G \geq H$, there is an $H^R$ with $G^R \geq H^R$; and since $H \geq G$, there is in turn a $G^{R'}$ with $H^R \geq G^{R'}$. Therefore
\[G^R \geq H^R \geq G^{R'},\]
and since $G$ has no dominated options, they must all be equal. This shows that $G^R = H^R$, whence by induction $G^R \cong H^R$. An identical argument shows that for each $H^R$, there is a $G^R$ with $H^R \cong G^R$, establishing the theorem.
\end{proof}

Theorem~\ref{theorem:deadendcanonical} has some remarkable consequences.

\begin{proposition}
If $G,H \in \mathcal{L}$, then $\birthday(G+H) = \birthday(G) + \birthday(H)$.
\end{proposition}

\begin{proof}
Assume $G$ and $H$ are in simplest form, and let $K$ be the simplest form of $G + H$. Then $\fbirthday(G+H) = \birthday(G) + \birthday(H)$. Now $K$ is obtained from $G + H$ by eliminating dominated options from subpositions of $G + H$. But Proposition~\ref{prop:de-fbirthday} shows that eliminating dominated options from $G + H$ never changes its formal birthday, so in fact $\fbirthday(K) = \birthday(G) + \birthday(H)$.
\end{proof}

\begin{proposition}
\label{prop:nonterminable-sums}
If $G,H \in \mathcal{L}$ with $G \neq 0$ and $H \neq 0$, then $0$ is not an option of (any form of) $G + H$.
\end{proposition}

\begin{proof}
A typical option of $G + H$ has the form $G^L + H$. Since $H \not\cong 0$, we have $G^L + H \not\cong 0$, and hence $G^L + H \neq 0$. By Theorem~\ref{theorem:deadendcanonical}, the simplest form of $G + H$ does not have $0$ as an option, and since $0$ is incomparable with every nonzero game, it follows that no form can.
\end{proof}


\subsection*{The Semilattice of Left Dead-Ends}

Now denote by $\mathcal{L}_n$ the set of (values of) Left dead-ends of birthday~$\leq n$. Viewed as a partial order, $\mathcal{L}_n$ (for $n \geq 1$) has two disconnected components: the singleton $\{0\}$ and its remainder, which we denote by $\mathcal{L}_n^\times = \mathcal{L}_n \setminus \{0\}$.

The partial order $\mathcal{L}_n^\times$ is easily seen to be a meet semilattice, with meet operation given by set union of Right options. Its maximal elements are the negative integers $\bar1$, $\bar2$, $\ldots$,~$\bar n$, and its unique least element is the \textbf{complete tree} $K_n = \combgame{\{\cdot|\mathcal{L}_{n-1}\}}$. The simplest form of $K_n$ is given recursively by
\[
M_0 = 0; \qquad M_{n+1} = \combgame{\{\cdot|0,M_n\}}.
\]
The games $M_n$ were first described in~\cite{LMNRS}, where they are called \textbf{perfect murders}.

The structures $\mathcal{L}_2^\times$ and $\mathcal{L}_3^\times$ are given in Figure~\ref{figure:dead-end-semilattices}. In Figure~\ref{figure:dead-end-semilattices} and throughout the sequel, we use a convenient abbreviated notation for elements of~$\mathcal{L}$. If $G,H,J \in \mathcal{L}$, then $GHJ$ denotes the game $\combgame{\{\cdot|G,H,J\}}$, and $G_\sh$ (pronounced ``$G$~sharp'') denotes the game $\combgame{\{\cdot|G\}}$.

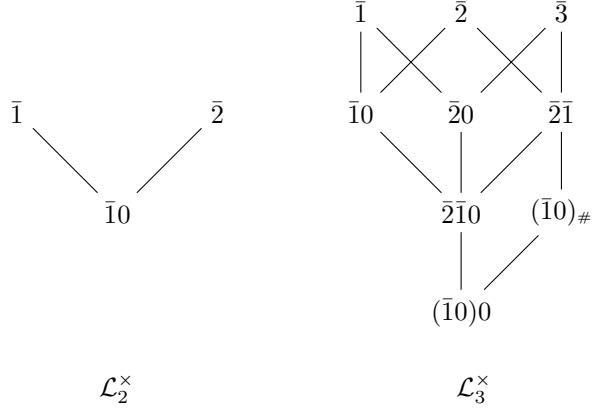
\begin{figure}
\begin{center}
\newcommand\tikzscale{1.3333cm}
\begin{tabular}{c@{\qquad\qquad}c}
\begin{tikzpicture}[x=\tikzscale,y=\tikzscale]
\node (10) at (0, 0) {$\bar{1}0$};
\node (1) at (-1, 1) {$\bar{1}$};
\node (2) at (1, 1) {$\bar{2}$};
\draw [-] (1) -- (10) -- (2);
\node (phantom) at (0, -1) {\color{white}$(\bar{1})$};
\end{tikzpicture} &
\begin{tikzpicture}[x=\tikzscale,y=\tikzscale]
\node (100) at (0, 0) {$(\bar{1}0)0$};
\node (210) at (0, 1) {$\bar{2}\bar{1}0$};
\node (10s) at (1, 1) {$(\bar{1}0)_\sh$};
\node (10) at (-1, 2) {$\bar{1}0$};
\node (20) at (0, 2) {$\bar{2}0$};
\node (21) at (1, 2) {$\bar{2}\bar{1}$};
\node (1) at (-1, 3) {$\bar{1}$};
\node (2) at (0, 3) {$\bar{2}$};
\node (3) at (1, 3) {$\bar{3}$};
\draw [-] (10) -- (1) -- (20) -- (3) -- (21) -- (2) -- (10) -- (210) -- (100) -- (10s) -- (21) -- (210) -- (20);
\end{tikzpicture}
\vspace{0.15in} \\
$\mathcal{L}_2^\times$ & $\mathcal{L}_3^\times$
\end{tabular}
\end{center}
\caption{The semilattices $\mathcal{L}_2^\times$ and $\mathcal{L}_3^\times$.\label{figure:dead-end-semilattices}}
\end{figure}

The cardinalities of $\mathcal{L}_n$ are easily computed up to $n = 5$.

\begin{center}
\begin{tabular}{c|c|c|c|c|c|c}
$n$ & 0 & 1 & 2 & 3 & 4 & 5 \bigstrut \\ \hline
$|\mathcal{L}_n|$ & 1 & 2 & 4 & 10 & 52 & 21278 \bigstrut
\end{tabular}
\end{center}

There is a remarkably simple recursive characterization of~$\mathcal{L}_n$. The partial order $\mathcal{L}_{n+1}$ consists of all antichains of~$\mathcal{L}_n$, with nonempty antichains ordered by set inclusion and the empty antichain disconnected. This shows in particular that $|\mathcal{L}_n|$ is always even for $n \geq 1$, since there is a one-to-one pairing between antichains with and without~$0$.

Moreover, we have the following unsurprising theorem.

\begin{theorem}
$\mathcal{L}_n^\times$ is a distributive meet semilattice.
\end{theorem}

One can define a ``distributive semilattice'' directly in a variety of ways, but we will simply prove the following, stronger theorem. Define the extension $\mathcal{L}_n^* = \mathcal{L}_n^\times \cup \{\triangledown\}$, where $\triangledown$ is a new symbol with $\triangledown \geq G$ for all $G \in \mathcal{L}_n^\times$. (Effectively, we are putting $0$ ``back into'' the partial order in its natural place.) Then:

\begin{theorem}
$\mathcal{L}_n^*$ is a distributive lattice.
\end{theorem}

\begin{proof}
Fix a value of $n \geq 1$ throughout this proof. Define
\begin{align*}
[G] &= \{X \in \mathcal{L}_{n-1} : X \geq G^R \text{ for some } G^R\} & \text{for all $G \in \mathcal{L}_n^\times$} \\
[\triangledown] &= \emptyset \vphantom{\Big(}
\end{align*}
It is easily seen that $G = \combgame{\{\cdot|[G]\}}$ for all $G \in \mathcal{L}_n^\times$, since $\combgame{\{\cdot|[G]\}}$ is obtained from $G$ by adding dominated Right options. Therefore
\begin{align*}
G \wedge H &= \combgame{\{\cdot \cgslash [G] \cup [H]\}} &\text{for all $G,H \in \mathcal{L}_n^\times$} \\
G \wedge \rlap{$\triangledown$}\phantom{H} &= G &\text{for all $G \in \mathcal{L}_n^*$}\vphantom{\Big(}
\end{align*}
Correspondingly, define for all $G,H \in \mathcal{L}_n^*$
\[
G \vee H = \begin{cases}
\triangledown & \text{if $[G] \cap [H] = \emptyset$}; \\
\combgame{\{\cdot|[G] \cap [H]\}} & \text{otherwise}.
\end{cases}
\]
Then we see that for all $G$ and $H$,
\[
[G \vee H] = [G] \cap [H] \quad\text{and}\quad [G \wedge H] = [G] \cup [H].
\]
Moreover, Theorem~\ref{theorem:deadendcanonical} implies that $G = H$ if and only if $[G] = [H]$, so that $G \mapsto [G]$ is an isomorphism from $\mathcal{L}_n^*$ onto a lattice of sets.
\end{proof}

\subsection*{Terminable Ends}

A particular class of Left dead-ends will be especially important to the succeeding analysis.

\begin{definition}
A Left dead-end $G$ is \textbf{terminable} if $0$ is a Right option of~$G$.
\end{definition}

We have seen that $0$ is incomparable with every nonempty Left dead-end. It follows that if \emph{some} form of $G$ is terminable, then so is \emph{every} form, so that terminability is an invariant property of the form of~$G$. Moreover, if $G$ and $H$ are nonempty, then $G + H$ is automatically non-terminable (Proposition~\ref{prop:nonterminable-sums}).

Now for any set of positive integers $\mathcal{N} \subset \mathbb{N}_{>0}$, consider the associated family of games
\[
\A = \{\bar{n}0 : n \in \mathcal{N}\}.
\]
A typical element of $\cl(\A)$ can be written as
\[
G \cong \overline{m} + \overline{n_1}0 + \overline{n_2}0 + \cdots + \overline{n_k}0
\]
with $m \in \mathbb{N}$ and $n_1,\ldots,n_k \in \mathcal{N}$. The sum $G$ can be terminable only in the case where there is a single nonempty summand, and we conclude that the terminable elements of $\cl(\A)$ are exactly
\[
\mathcal{T} = \mathcal{A} \cup \{\bar{1}\}.
\]
It follows that in the universe $\mathcal{D}(\mathcal{A})$, the terminable Left ends are exactly the set~$\mathcal{T}$, and we obtain an infinite (indeed, uncountable) family of subuniverses of~$\mathcal{E}$.

\section{Comparison in Subuniverses of \E}
\label{section:comparison}

We now consider the structure of the general absolute universe $\U \subset \E$. By Propositions \ref{proposition:endsgenerate} and~\ref{proposition:nonewends}, $\U$ is characterized by the set $\A$ of its Left ends.

For any set of Left dead-ends $\A$, denote by $\up(\A)$ the \textbf{upward closure} of~$\A$:
\[
\up(\A) = \{G \in \mathcal{L} : G \geq H \text{ for some } H \in \cl(\A)\}.
\]
If $\U = \D(\A)$ is a dead-ending universe, then we write $\up(\U) = \D(\up(\A))$.

Now if $\up(\U) = \up(\V)$, then \U-strong and \V-strong coincide, and it follows from Theorem~\ref{theorem:lnsequiv} that $\geq_\U$ and $\geq_\V$ are the same relation: $G \geq_\U H$ if and only if $G \geq_\V H$, for all $G$ and~$H$. The major aim of this section is to prove a converse: if $\geq_\U$ and $\geq_\V$ agree on~$\mathcal{D}$ (i.e., if they have identical restrictions to the dicot universe), then in fact $\up(\U) = \up(\V)$. The proof is not difficult, but the result is suprising and has strong implications.

The upward closure appears to expand $\A$ over a potentially infinite domain, but it is actually a quite manageable structure. From Proposition~\ref{prop:de-fbirthday} it follows that for any~$H$, there are just finitely many $G \geq H$; thus finitely generated universes have finitely generated upward closures. Moreover, note that $\up(\A)$, as defined, is already simply closed:

\begin{proposition}
\label{prop:upward-closure-commutes}
$\up(\A)$ is simply closed, for any set of Left dead-ends~\A.
\end{proposition}

\begin{proof}
To see that $\up(\A)$ is hereditarily closed, let $G \in \up(\A)$, with $G \geq H$ for $H \in \cl(\A)$. Then for every Right option~$G^R$, we have $G^R \geq$ some~$H^R$. Since $\cl(\A)$ is closed, $H^R \in \cl(\A)$, whence $G^R \in \up(\A)$.

Additive closure follows trivially from the fact that $\geq$ respects addition.
\end{proof}

\subsection*{Comparison and Upward Closure}

The familiar \textbf{adjoint} will be a crucial tool. Recall~\cite{Sie13}~\cite{Sie15pmcf}:

\begin{definition}
\label{definition:adjoint}
The \textbf{adjoint} of~$G$, denoted by~$G^\circ$, is defined by
\[
G^\circ = \begin{cases}
\mathspoofwidth{\cgtwo{(G^R)^\circ}{(\fatGL)^\circ}}{\cgstar} & \text{if $G \cong 0$}; \\
\cgtwo{(G^R)^\circ}{\mathrlap{0}\phantom{(\fatGL)^\circ}} & \text{if $G \not\cong 0$ and $G$ is a Left end}; \\
\cgtwo{\phantom{(G^R)^\circ}\mathllap{0}}{(\fatGL)^\circ} & \text{if $G \not\cong 0$ and $G$ is a Right end}; \\
\cgtwo{(G^R)^\circ}{(\fatGL)^\circ} & \text{otherwise}.
\end{cases}
\]
\end{definition}

\begin{proposition}
\label{prop:adjoint-p}
$o(G+G^\circ)$ is a \P-position, for all $G \in \mathcal{M}$.
\end{proposition}

\begin{proof}
It suffices (by symmetry) to show that $o(G+G^\circ) \geq \P$. Every Right option of $G+G^\circ$ reverts to a position of the form $G^L+(G^L)^\circ$ or $G^R+(G^R)^\circ$, except for Right's move to $G+0$ when $G$ is a Left end. But in that case, $o(G) \geq \N$ \emph{a priori}.
\end{proof}

When $G$ is a Left dead-end, the adjoint reduces to
\[
G^\circ = \begin{cases}
\mathspoofwidth{\cgtwo{(G^R)^\circ}{0}}{\cgstar} & \text{if $G \cong 0$}; \\
\cgtwo{(G^R)^\circ}{0} & \text{otherwise}.
\end{cases}
\]
Closely related is the \textbf{Left-modified adjoint} $G^\oast$, defined only for Left ends, which we will need for technical reasons:

\begin{definition}
Let $G$ be a Left dead-end. The \textbf{Left-modified adjoint} of~$G$, denoted by~$G^\oast$, is given by:
\[
G^\oast = \begin{cases}
\cgtwo{\mathspoofwidth{(G^R)^\circ}{0}\vphantom{(G^R)^\circ}}{\cgstar} & \text{if $G \cong 0$}; \\
\cgtwo{(G^R)^\circ}{\cgstar} & \text{otherwise}.
\end{cases}
\]
\end{definition}
Note that in the case $G \not\cong 0$, the Left options of $G^\oast$ have the form $(G^R)^\circ$ (ordinary adjoint), \emph{not}~$(G^R)^\oast$.

\begin{figure}
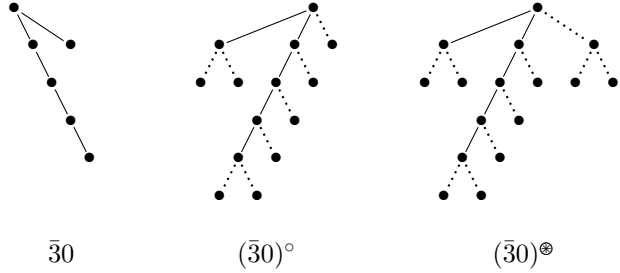

\begin{center}
\begin{tabular}{c@{\qquad\quad}c@{\qquad\quad}c}
$\letree{\en(|\en(|\en(|\en(|\en(|{\color{white}\en}[linestyle=white]))))\en)}$ &
$\letree{\en(\en(\en(\en(\en(\en[linestyle={dotted,thick}]|\en[linestyle={dotted,thick}])|\en[linestyle={dotted,thick}])|\en[linestyle={dotted,thick}])|\en[linestyle={dotted,thick}])++\en(\en[linestyle={dotted,thick}]|\en[linestyle={dotted,thick}])|\en[linestyle={dotted,thick}])}$ &
$\letree{\en(\en(\en(\en(\en(\en[linestyle={dotted,thick}]|\en[linestyle={dotted,thick}])|\en[linestyle={dotted,thick}])|\en[linestyle={dotted,thick}])|\en[linestyle={dotted,thick}])++\en(\en[linestyle={dotted,thick}]|\en[linestyle={dotted,thick}])|++\en[linestyle={dotted,thick}](\en[linestyle={dotted,thick}]|\en[linestyle={dotted,thick}]))}$
\vspace{0.15in}\\
$\bar{3}0$ & $(\bar{3}0)^\circ$ & $(\bar{3}0)^\oast$
\end{tabular}
\end{center}
\caption{The Left end $\bar{3}0$, pictured together with its adjoint and Left-modified adjoint. For clarity, dotted lines are used to indicate the ``adjoined'' structure that was not present in the original tree.}
\end{figure}

The importance of adjoints to the theory of dead-ends is captured by the following lemma.

\begin{lemma}
\label{lemma:adjoint-strength}
Let $G$ and $H$ be Left dead-ends. Then
\[
G \geq H \quad\text{if and only if}\quad o(G^\circ + H) \leq \P.
\]
\end{lemma}
(Compare: in normal play, $G \geq H$ if and only if $o(-G + H) \leq \P$.)

\begin{proof}
If $G \geq H$, then we have immediately that
\[
o(G^\circ + H) \leq o(G^\circ + G) = \P.
\]

Conversely, suppose that $o(G^\circ + H) \leq \P$. If $G \cong 0$, then $G^\circ = *$, so that $o(H) \leq \N$. Since $H$ is a Left dead-end, this implies $H \cong 0$ as well.

Therefore, assume instead that $G \not\cong 0$. Then to prove $G \geq H$, we must show that for every~$G^R$, there is an $H^R$ with $G^R \geq H^R$ (Theorem~\ref{theorem:de-comparison}). But for each such~$G^R$, the adjoint $(G^R)^\circ$ is a Left option of~$G^\circ$, so that
\[
o((G^R)^\circ + H) \leq \N,
\]
that is, Right has a winning move on $(G^R)^\circ + H$. Right's winning move cannot be of the form $(G^{RL})^\circ + H$, since $G^R$ has no Left option; nor can it be to $0 + H$, since $H$ is a Left end. Therefore it must be to some $(G^R)^\circ + H^R$, whence by induction $G^R \geq H^R$.
\end{proof}

We arrive at the following Theorem, connecting the upward closure of $\A$ with dicot comparison in the universe $\D(\A)$.

\begin{theorem}
\label{theorem:upward-closure-comparison}
Let $\U = \D(\A)$ be a dead-ending universe. Then the following are equivalent, for every Left dead-end~$G$.
\begin{enumerate}
\item[(i)] $G \not\in \up(\A)$;
\item[(ii)] $G^\circ$ is Left \U-strong;
\item[(iii)] $G^\oast \geq_\U 0$.
\end{enumerate}
\end{theorem}

\begin{proof}
(i) $\Rightarrow$ (ii): Suppose $G \not\in \up(\A)$. Then for every Left end $X \in \U$, we have $G \not\geq X$, so by Lemma~\ref{lemma:adjoint-strength} it follows that $o(G^\circ + X) \geq \N$.

\vspace{0.1in}\noindent
(ii) $\Rightarrow$ (iii): Since $G^\circ$ is Left \U-strong and $G^\circ$ and $G^\oast$ have the same Left options, it follows that $G^\oast$ is also Left \U-strong. Thus it suffices to verify the maintenance property. But the only Right option of $G^\oast$ is $\cgstar$, which reverts to~$0$. Certainly $0 \geq_\U 0$, and the conclusion follows.

\vspace{0.1in}\noindent
(iii) $\Rightarrow$ (i): If $G^\oast \geq_\U 0$, then $G^\oast$ must be \U-strong, so $G^\circ$ is \U-strong. Therefore $o(G^\circ + X) \geq \N$ for all $X \in \cl(\A)$, which implies $o(G^\circ + X) \geq \N$ for all $X \in \up(\A)$. In particular, since $o(G^\circ + G) \leq \P$, we have $G \not\in \up(\A)$.
\end{proof}







It follows that $\up(\A)$ is recoverable from the restriction of $\geq_\U$ to dicots:
\[
\up(\A) = \{G \in \mathcal{L} : G^\oast \not\geq_\U 0\}.
\]
The following facts are now apparent. Consider a partial order $\succcurlyeq$ of the dicot universe~\D. Say that $\succcurlyeq$ is \textbf{permissible} if there is some dead-ending universe~$\U$ for which $\succcurlyeq$ is the restriction of $\geq_\U$ to~\D. Then:
\begin{itemize}
\item There is a one-to-one correspondence between upward-closed sets of Left dead-ends and permissible dicot orderings.
\item For any dead-ending universe \U, the relation $\geq_\U$ is determined entirely by its restriction to~\D.
\end{itemize}





\subsection*{The Hierarchy of Partial Orders}

In Section~\ref{section:left-ends}, we showed that there is an infinite hierarchy of universes $\U$ with $\D \subset \U \subset \E$. Now Theorem~\ref{theorem:upward-closure-comparison} shows that every such $\U$ imposes the same order-relations as its upward closure $\up(\U)$, and therefore embeds as a submonoid of $\up(\U)$. In light of Theorem~\ref{theorem:upward-closure-comparison}, it makes sense to regard upward-closed universes as the ``fundamental'' elements of the hierarchy, in the sense that the others introduce no new structure.

We might instead consider a hierarchy of permissible dicot orderings, each a refinement of $\geq_\D$ and a coarsening of $\geq_\E$. This hierarchy is, of course, isomorphic (by Theorem~\ref{theorem:upward-closure-comparison}) to the corresponding hierarchy of upward-closed universes.

Note that the uncountably many universes introduced in Section~\ref{section:left-ends} all have distinct upward closures, and hence, distinct orderings $\geq_\U$. To see this, let $\U$ and $\mathcal{V}$ be any two such universes; we can assume without loss of generality that there is some $\bar{n}0 \in \U$ with $\bar{n}0 \not\in \mathcal{V}$. Then the terminable elements of $\mathcal{V}$ (other than~$\bar{1}$) all have the form $\bar{k}0$ for $k \neq n$. But $\bar{n}0 \not\geq \bar{1}$ and $\bar{n}0 \not\geq \bar{k}0$, and also $\bar{n}0 \not\geq G$ for any non-terminable~$G$. It follows that $\bar{n}0 \not\in \up(\mathcal{V})$.

In particular, this establishes the existence of an uncountable family of distinct orderings $\geq_\U$.

\subsection*{Constructive Comparison}

As we have seen, there are uncountably many distinct relations~$\geq_\U$: every set $\mathcal{N}$ of natural numbers is recoverable from dicot comparisons modulo some universe~\U. We therefore cannot hope to find algorithmic constructions for all of them. Nonetheless, we can show that whenever the Left ends of $\U \subset \E$ are strongly computable (in a sense to be defined shortly), then so is the relation~$\geq_\U$.

\begin{definition}
Let $G$ be any partizan game. The \textbf{$n^\text{th}$ truncation} $\tau_n(G)$ is defined by:
\[
\tau_0(G) = 0; \qquad \tau_{n+1}(G) = \combgame{\{\tau_n(G^L)|\tau_n(G^R)\}}.
\]
\end{definition}

$\tau_n(G)$ is equivalent to playing on~$G$ and stopping immediately after the $n^\text{th}$ move.

\begin{definition}
\label{definition:testset}
Let $\mathcal{A}$ be a set of Left dead-ends. The \textbf{$n^\text{th}$ test set} $T_n(\mathcal{A})$ is the set of minimal elements of $\{\tau_n(G) : G \in \cl(\mathcal{A})\}$.

If $\U = \D(\mathcal{A})$ is a dead-ending universe with Left ends~$\mathcal{A}$, then we write $T_n(\U)$ as a synonym for $T_n(\mathcal{A})$.
\end{definition}

(In Definition~\ref{definition:testset}, the ``minimal elements'' are taken according to the absolute ordering of Left dead-ends. Since $\{\tau_n(G) : G \in \cl(\mathcal{A})\} \subset \mathcal{M}_n$, it is always finite, so the set of its minimal elements is nonempty.)

\begin{examples}
The test sets for some familiar universes are easily described:
\begin{itemize}
\item $T_n(\D) = \{0\}$.
\item $T_n(\E) = \{0,M_n\}$.
\item $T_n(\D(\bar{1})) = \{0,\bar{1},\ldots,\bar{n}\}$.
\item $T_n(\D(\bar{1}0)) = \{\bar{1}0,\ 2\cdot(\bar{1}0),\ \ldots,\ k\cdot(\bar{1}0)\}$ if $n = 2k$;

${\color{white}T_n(\D(\bar{1}0)) = }~\{\bar{1}0,\ 2\cdot(\bar{1}0),\ \ldots,\ k\cdot(\bar{1}0),\ k\cdot(\bar{1}0) + \bar{1}\}$ if $n = 2k+1$.
\end{itemize}
(The case $T_n(\D(\bar{1}0))$ will be discussed in more detail in Section~\ref{section:domineering}.)
\end{examples}

\begin{theorem}
\label{theorem:de-truncation}
Let $\U$ be a dead-ending universe and let $G \in \mathcal{M}$. Then $G$ is Left \U-strong if and only if $o(G + X) \geq \N$ for all $X \in T_n(\U)$, where $n = \fbirthday(G)$.
\end{theorem}

The intuition behind Theorem~\ref{theorem:de-truncation} is that, in any play of $G + X$, Left is constrained to play on~$G$. Therefore, if Left can arrange to leave herself without a move on $G + X$, she can do so in at most $\fbirthday(G)$ moves; thus Right will make at most $\fbirthday(G)$ moves on~$X$. First a lemma:

\begin{lemma}
\label{lemma:de-truncation}
Let $G$ be any partizan game and let $X$ be a Left dead-end. If $o(G + X) \leq \P$, then also $o(G + \tau_n(X)) \leq \P$, for all $n \geq \fbirthday(G)$.
\end{lemma}

\begin{proof}
Certainly $G$ cannot be a Left end (otherwise $o(G + X) \geq \N$ would be automatic), so it suffices to show every $o(G^L + \tau_n(X)) \leq \N$, i.e., Right wins moving first on $G^L + \tau_n(X)$.

But $o(G^L + X) \leq \N$, so either $o(G^{LR} + X) \leq \P$, or $o(G^L + X^R) \leq \P$, or else $G^L + X$ is a Right end. In the first case, $o(G^{LR} + \tau_n(X)) \leq \P$ by induction. In the second case, we have $o(G^L + \tau_{n-1}(X^R)) \leq \P$ by induction, since $\fbirthday(G^L) \leq n-1$; but $\tau_{n-1}(X^R)$ is a Right option of $\tau_n(X)$. Finally, if $G^L + X$ is a Right end, then so is $G^L + \tau_n(X)$.
\end{proof}

\begin{proof}[Proof of Theorem~\ref{theorem:de-truncation}]
The forward direction is immediate. For the converse, suppose (for contrapositive) that Left is not \U-strong. Then there is some Left dead-end $X \in \U$ with $o(G + X) \leq \P$. By Lemma~\ref{lemma:de-truncation}, $o(G + \tau_n(X)) \leq \P$. This proves the Theorem, since $Y \leq \tau_n(X)$ for some $Y \in T_n(\U)$.
\end{proof}

Thus in Theorem~\ref{theorem:lnsequiv}, we may substitute the finite test from Theorem~\ref{theorem:de-truncation} in place of \U-strongness, and in particular:

\begin{corollary}
\label{corollary:computable}
If the sequence $n \mapsto T_n(\U)$ is computable, then so is~$\geq_\U$.
\end{corollary}

\begin{example}
Let $\U = \D(\bar{1})$. Then $G$ is Left \U-strong if and only if $o(G + \bar{n}) \geq \N$, for all $n \leq \fbirthday(G)$.

There is no avoiding the fact that we need to perform $n$ separate strongness tests. For example, if $G = \bar{n}^\circ$, then $o(G + \bar{k}) \geq \N$ for all $k \neq n$, but ${o(G + \bar{n}) = \P}$.
\end{example}

\begin{corollary}
If $\U \subset \E$ is finitely generated, then $\geq_\U$ is computable.
\end{corollary}

\begin{proof}
Let $\A$ be a finite set of Left dead-ends with $\U = \D(\A)$.
Now certainly $\tau_n(G + H) = \tau_n(\tau_n(G) + \tau_n(H))$ for all $G$ and~$H$, so the test set $T_n(\A)$ can be computed by iteratively taking sums and truncating, until a closure is formed (which must happen eventually, since $T_n(\A)$ is finite).
\end{proof}

\subsection*{Games Born by Day 2}

We will discuss simplest forms modulo the general dead-ending universe in the next section. However, we already have enough theory to brute-force a computation of games born by day 2 in each of the dead-ending universes discussed so far. There are $232$ nonisomorphic dead-ending game trees born by day~2; thus we can brute-force the number of \emph{distinct} games born by day 2 with at most $232^2$ comparisons.

The results are summarized in Figure~\ref{figure:de-day2}.
\begin{figure}
\begin{center}
\begin{tabular}{c|cccr}
$\U$ & $\mathcal{E}$ & $\mathcal{D}(\bar{1}0)$ & $\mathcal{D}(\bar{1})$ & \multicolumn{1}{c}{$\mathcal{D}$} \bigstrut \\
\hline
Total game trees in \U & 232 & 232 & 230 & 10 \bigstrut[t] \\
Distinct games in $\mathcal{U}$ (mod \U) & 196 & 196 & 194 & 9 \bigstrut[t]
\end{tabular}
\end{center}
\caption{Distinct games born by day 2 in various dead-ending universes. \label{figure:de-day2}}
\end{figure}
For each universe $\D$, $\D(\bar{1})$, $\D(\bar{1}0)$, and~$\E$, we show the number of game trees of height~2 in that universe, and the number that are distinct modulo~\U. The numbers for $\D(\bar{1}0)$ and $\E$ are identical, which is expected: $\D(\bar{1}0)$ already contains all the rank~2 dead-ends in~\E. Indeed, \emph{every} dead-ending universe $\U \neq \D$ will look exactly like $\D(\bar{1})$ or $\D(\bar{1}0)$ on day~2, depending only upon whether $\bar{1}0 \in \U$.

Dorbec et al~\cite{DRSS15} showed that there are exactly 1268 distinct elements of $\D$ born by day~3 and 7541 distinct games in $\mathcal{M}$ (mod \D) born by day~3. Obtaining exact day-3 counts for larger universes is presumably hopeless, but it may be interesting to investigate bounds on such counts.

\begin{openproblem}
Obtain bounds on the number of games born by day~3 modulo various dead-ending universes.
\end{openproblem}

\section{Simplest Forms in Mis\`ere Universes}
\label{section:invariant-forms}

We now consider the question of \textbf{simplest} or \textbf{canonical forms} modulo a general universe. We will shortly introduce a general notion of simplest form that works for any universe~$\U$ (including when $\U$ is not dead-ending), and that is computable whenever $\geq_\U$ is computable (and in particular, for any $\U \subset \E$ that is computable in the sense of Corollary~\ref{corollary:computable}). In order to motivate our approach, we first survey the previous developments in this area.

The idea of a simplest form for a game $G$ was introduced most famously, and most successfully, in the normal-play partizan theory of Berlekamp, Conway, and Guy~\cite{BCG01}. The dramatic success of the classical theory naturally led to a desire to seek out simplest forms in other contexts. Not long after the discovery of the normal-play theory, Conway showed that impartial games also admit simplest forms in mis\`ere play~\cite{Con01}. By combining these two classical results, the present author later extended them to partizan games in mis\`ere play~\cite{Sie15pmcf}.

In each of these cases, it can be shown that each game $G$ admits a unique \textbf{simplest form}~$K$, whose game tree has strictly fewer edges than any other form of~$G$. Although the details vary, the essence of the proof is the same in each case: dominated options can be eliminated; reversible options can be bypassed; and a game with neither must be in simplest form.

\subsection*{Simplest forms in \D}

In mis\`ere universes other than~$\mathcal{M}$, the situation becomes less clear. It is still true (as we shall see) that a game with no dominated or reversible options must be in simplest form. The trouble is that in the general universe, there may be reversible options that cannot be bypassed, in the specific case where they reverse \emph{through} an end. That this situation is problematic has been understood, in various guises, at least since Ettinger's work in the mid-1990s~\cite{Ett96}. Following Ettinger, we will call a Left option \textbf{atomic reversible} if it reverses through a Left end; \textbf{ordinary reversible} if it reverses through a game that is not a Left end.

For a typical example, consider the dicot universe~$\D$ and the game
\[
G = \combgame{\{0,\cgstar|\cgstar\}}.
\]
It is easy to check that $G \geq_\D 0$ (using Theorem~\ref{theorem:dicot-comparison}, say), so that Left's option $G^L = \cgstar$ is reversible through~$0$, a Left end. But bypassing it would yield the ``reduction''
\[
G' = \combgame{\{0|\cgstar\}},
\]
which fails to satisfy the proviso: $o(G') \not\geq \N$, so $G' \not\geq_\D 0$ (again cf.~Theorem~\ref{theorem:dicot-comparison}). The original game $G$ is \D-strong, but $G'$ is not, so the Left option $\cgstar$ cannot be removed; it is the only witness to the fact that $G$ is \D-strong.

In \cite{DRSS15} it is shown that in~\D, one can simply (and, in hindsight, somewhat magically) power through these difficulties. The main idea is that, rather than \emph{bypass} an atomic reversible~$G^L$, one can simply \emph{replace} it by a simpler reversible option. For an example of this procedure, consider the game
\[
H = \combgame{\{0,\{\cgstar||0,\{\cgstar|0\}\}|||\cgstar\}}.
\]
It looks like a mess, but the point is that Left's option $H^L = \combgame*{\{\cgstar||0,\{\cgstar|0\}\}}$ reverses through~$0$. It cannot be eliminated, for the same reason as before, but it \emph{can} be replaced by the simpler reversible option~$\cgstar$, so that in fact $H \equiv_\D G$. This procedure is perhaps visualized most easily in terms of the game trees of $G$ and~$H$, as in Figure~\ref{figure:dicot-reduction}.

\begin{figure}
\centering
\cgtree[unit=0.5cm]{\bullet(+\bullet(\cgstar(0|0)|0\bullet(\cgstar(0|0)|0))0|+++\cgstar(0|0))}
\raisebox{3cm}{$\qquad \Longrightarrow \qquad$}
\raisebox{2cm}{\cgtree[unit=0.5cm]{\bullet(+\cgstar(0|0)0|+\cgstar(0|0))}}
\caption{\label{figure:dicot-reduction} A typical example of dicot reduction.}
\end{figure}

Now since $0$ is the only end in~\D, any $G$ with an atomic reversible $G^L$ must necessarily satisfy $G \geq_\D 0$. It follows\footnote{Some effort is required to work out the details; see~\cite{DRSS15}.} that \emph{any} such $G^L$ can be replaced by~$\cgstar$. Because $G^L$ is reversible, it will never contribute to the maintenance property in comparisons involving~$G$: thus the ``fine structure'' of $G^L$ is immaterial.

Simplest forms for~$\D$ are thus obtainable: we must \emph{eliminate} dominated options; \emph{bypass} ordinary reversible ones; \emph{replace} atomic reversible options by~$\cgstar$; and, finally, \emph{erase} extraneous instances of~$\cgstar$ (those that are reversible, but are not actually necessary in order to ensure that $G$ is \D-strong). It turns out that if $G \in \D$ is fully reduced, in the sense that none of these four operations are possible, then $G$ is indeed in simplest form (modulo~\D).

This shows that the situation for $\D$ is almost as nice as for the preceding cases. There is just one minor, but still niggling, quibble. In all the preceding cases (normal play, impartial mis\`ere, full mis\`ere) the simplest forms obtained are maximally information-dense: every edge of the game tree is indispensable; there is some compound in which it is the \emph{only} winning move. But consider the Left option $G^L = \cgstar$ in the game $G = \combgame{\{0,\cgstar|\cgstar\}}$. Its \emph{Right} option $G^{LR} = 0$ is necessary for the proviso, but it also has a \emph{Left} option $G^{LL} = 0$. This Left option serves no combinatorial purpose; it is only necessary to ensure that $G \in \D$.

At first this seems a pedantic objection---one extraneous edge!---but it is nonetheless unsettling, and it is a harbinger of greater difficulties to come.

\subsection*{Simplest forms in \E}

Recent work of Larsson et al \cite{LMNRS} has extended the preceding results to obtain canonical forms for the full mis\`ere universe~\E. The procedure is the same as for~\D, but with atomic reversible Left options replaced by games of the form $\combgame{\{\cdot|M_n\}}$, rather than~$\cgstar$.

To see that this works, suppose that $G$ has an atomic reversible~$G^L$, so that $G \geq_\E G^{LR}$ for some Left end~$G^{LR}$. Then it must be the case that $G \geq_\E M_n$ for some~$n$, whereupon $\combgame{\{\cdot|M_n\}}$ suffices as a replacement for~$G^L$. We can ensure the reduction is canonical by giving preference to the \emph{minimal} such~$n$.

As an example, consider the game
\[
H = \combgame{\{0,\{0,\bar{1}|M_5\}||\bar{2},0\}}.
\]
It is not hard to show that $H \geq_\E M_5$, so that Left's option $\combgame{\{0,\bar{1}|M_5\}}$ is reversible. It follows that $H$ is equal to
\[
H' = \combgame{\{0,\{\cdot|M_3\}||\bar{2},0\}},
\]
which turns out to be the canonical form of~$H$, in the sense defined above.

Observe the subtle change of wording: in the preceding cases, we spoke of the ``simplest'' form for~$G$; in the case of \E, a ``canonical'' form. Often in the literature these terms have been used interchangeably, but there is a key difference. The form $H'$ is indeed canonical, in the sense that applying the same operations to any other form of $H$ will yield an identical result. But it is not, in fact, simplest: the following game, whose tree has one fewer edge than~$H'$, is also equal to~$H$.
\[
H'' = \combgame{\{0,\{\cdot|\bar{2}0\}||\bar{2},0\}}.
\]

Moreover, the problem of ``extraneous edges'' that we observed in the case of $\D$ has become much worse. The game tree for $M_3$ (say) has 5 edges---a cumbersome way of expressing a single bit of information (``Yes, $H$~is Left \E-strong''). In order to satisfy the maintenance property, it would suffice to eliminate $H^L$ entirely; in order to satisfy the proviso, it would suffice to replace $H^L$ by $\bar{1} = \combgame{\{\cdot|0\}}$. But in order to satisfy both at the same time, we are forced into the uneasy compromise of $\combgame{\{\cdot|M_3\}}$. Thus a substantial part of the canonical form of~$H$ (the subtree~$M_3$) serves no combinatorial purpose.

\subsection*{A More Complicated Universe}

It is instructive to consider one further example, which shows that in the general universe the situation is still worse. Let
\[
A = 0\bar{1}\bar{4}\bar{5}; \qquad B = 0\bar{2}\bar{3}\bar{5}; \qquad \mathcal{U} = \D(A,B),
\]
and
\[
G = \combgame{\{0,\{\cdot|A\},\{\cdot|B\}||0,\bar{5}\}}.
\]
It is easily seen that $G$ is \U-strong (in fact \E-strong): if $X \in \E$ is a Left dead-end, then from $G + X$ Left has a winning move to $\combgame{\{\cdot|A\}} + X$ (or $\combgame{\{\cdot|B\}} + X$).

So we have $G \geq_\U A$ and $G \geq_\U B$ (the maintenance property is easily verified), and both of the options $\combgame{\{\cdot|A\}}$ and $\combgame{\{\cdot|B\}}$ are reversible through a Left end. We can therefore eliminate either of them---but not both, for then $G$ would cease to be \U-strong. \emph{So, which one?} In order to define ``canonical forms'' in \U, we must specify a replacement game for atomic reversible Left options, analogous to $M_n$ in the preceding example.

But $A$ and $B$ have the same birthday and number of options, and their game trees are identical in size. There is nothing that really distinguishes them in any way that could meaningfully be called ``canonical''. Nor will anything else in $\U$ suffice as a replacement game, since $A$, $B$, and $\bar{1}$ are the only terminable Left ends in~\U.

Examples like this suggest that even the notion of ``canonical form'' is suspect in the general universe. It is conceivable that we could press forward and contrive some distinction between $A$ and $B$ that makes one or the other ``canonical,'' but as we consider still more complicated universes, we will be forced into increasingly arbitrary choices.

\subsection*{A Path Forward}

The preceding examples suggest that for a general universe $\U$ and a $G \in \U$, we cannot expect to find a simplest form for~$G$---or even a ``canonical form'' in any meaningful sense---within~\U. It turns out, however, that we can obtain a simplest form $K$ for $G$ in a quite natural way, if we simply drop the expectation that $K \in \U$.

What, after all, are the essential properties of simplest forms, in those contexts where they arise? They are invariant of the form of~$G$; they respect addition; and they capture the exact information necessary to determine $o(G+X)$ for any $X \in \U$. It just so happens that in normal play and full mis\`ere, they \emph{also} reside within the universe $\U$ itself---which is quite convenient, but there is no reason to expect this to be true in general. Indeed, in many other contexts (e.g., the Smith--Fraenkel theory of loopy impartial games~\cite{Smi66}; the theory of stopper-sided games~\cite{BCG01}; mis\`ere quotients~\cite{Pla05}~\cite{PS08}), the ``simplest forms'' take some other shape, yet the theories work just fine.

We propose a similar approach here. We will define a notion of an \textbf{augmented game}~$G$, in which each subposition $H$ may have, in addition to its ordinary options $H^L$ and~$H^R$, an additional ``tombstone'' Left option that we denote by~$\Sigma^L$, and/or a ``tombstone'' Right option~$\Sigma^R$. The option $\Sigma^L$ functions as a ``generic witness'' to the fact that the subposition $H$ is Left \U-strong. It is a way of saying: ``it is important to preserve the fact that $H$ is strong, but the specific details of \emph{why} it is strong are not relevant to the value of~$H$''---because those details are atomic reversible.

In the expanded universe of augmented games, we can bypass reversible options just as we ordinarily would, but in the case where $G^L$ is atomic reversible, we must leave behind $\Sigma^L$ as a ``tombstone.'' In the example $G$ given above, the simplest form is then simply
\[
G = \combgame{\{0,\Sigma^L|0,\bar{5}\}}.
\]
The reversible options are replaced by the generic witness $\Sigma^L$.

We now prove that the games so obtained have all the properties one would expect of simplest forms.

\begin{definition}
The universe $\Mhat$ of \textbf{augmented games} is defined as follows.
\begin{align*}
\Mhat_0 &= \{0\} \\
\Mhat_{n+1} &= \Big\{\combgame{\{\mathscr{G}^L~|~\mathscr{G}^R\}} : \mathscr{G}^L \subset \Mhat_n \cup \big\{\Sigma^L\big\},\ \mathscr{G}^R \subset \Mhat_n \cup \big\{\Sigma^R\big\}\Big\} \\
\Mhat &= \bigcup_n \Mhat_n
\end{align*}
Here $\Sigma^L$ and $\Sigma^R$ are new symbols that are distinct from all elements of $\Mhat$ (and are not themselves elements of~$\Mhat$).
\end{definition}
Note that the only distinction between $\mathcal{M}$ and $\Mhat$ is the possible appearance of the tombstones $\Sigma^L$ and~$\Sigma^R$, and in particular, we have $\mathcal{M}_n \subset \Mhat_n$ for all~$n$. Also note that there are only two tombstone symbols (globally): $\Sigma^L$~denotes a single specific entity, which may or may not be present in any particular position.

We will now show that for every $G \in \U$, there is a unique simplest augmented game $K \in \Mhat$ with $G \equiv_\U K$. The construction is the same in all universes, differing only in the evaluation of~$\geq_\U$.

If $G \in \Mhat$, we say that $G$ \textbf{has $\Sigma^L$} if $\Sigma^L$ is a Left option of~$G$; similarly for~$\Sigma^R$. We call $\Sigma^L$ and $\Sigma^R$ (if present) the \textbf{tombstone options} of~$G$; all other options are \textbf{ordinary options}. The variables $G^L$ and $G^R$ will always be understood to range over the \emph{ordinary} Left and Right options of~$G$, and it will also be convenient to write $G^\mathcal{L}$ and $G^\mathcal{R}$ for, respectively, the sets of ordinary Left and Right options of~$G$.

If $G$ has $\Sigma^L$, then it enjoys many of the properties of a Left end. This motivates the following definition.
\begin{definition}
Let $G \in \Mhat$. We say that $G$ is \textbf{Left (resp.~Right) end-like} if $G$ is a Left (resp.~Right) end or $G$ has~$\Sigma^L$ (resp.~$\Sigma^R$).
\end{definition}

\begin{definition}
Let $G \in \Mhat$. We define the \textbf{(mis\`ere) outcomes} $o^L(G)$ and $o^R(G)$ by
\begin{align*}
o^L(G) &= \begin{cases}
\mathscr{L} & \text{if either $G$ is Left end-like, or else some $o^R(G^L) = \mathscr{L}$;} \\
\mathscr{R} & \text{otherwise};
\end{cases} \\
o^R(G) &= \begin{cases}
\mathscr{R} & \text{if either $G$ is Right end-like, or else some $o^L(G^R) = \mathscr{R}$;} \\
\mathscr{L} & \text{otherwise}.
\end{cases}
\end{align*}
\end{definition}
Thus if $G$ has $\Sigma^L$, then necessarily $o^L(G) = \L$; likewise for~$\Sigma^R$. (This is because $\Sigma^L$ is interpreted as the ``tombstone'' of an atomic reversible option.)

\begin{definition}
Let $G,H \in \Mhat$. Then the \textbf{disjunctive sum} $G + H$ is defined as follows. The ordinary options are given in the usual manner:
\begin{samepage}
\[
(G+H)^\mathcal{L} = \big\{G^L + H,\ G + H^L\big\}; \qquad (G+H)^\mathcal{R} = \big\{G^R + H,\ G + H^R\big\}.
\]
For the tombstone options: $G+H$ has $\Sigma^L$ if and only if $G$ and $H$ are \emph{both} Left end-like, and \emph{at least one} of them has~$\Sigma^L$; likewise for~$\Sigma^R$.
\end{samepage}
\end{definition}

It is easily verified that $\Mhat$ is a commutative monoid, with $\mathcal{M} \subset \Mhat$.

With outcomes and disjunctive sum so defined, the definitions of $\geq_\U$ and $\U$-strong extend to all $G,H \in \Mhat$ in the usual manner:
\[
G \geq_\U H \quad\text{if}\quad o(G + X) \geq o(H + X) \quad\text{for all}\quad X \in \U
\]
\[
G \text{ is \textbf{Left \U-strong}} \quad\text{if}\quad o(G + X) \geq \mathscr{N} \quad\text{for every Left end}\quad X \in \U
\]
Here and throughout the sequel, it is understood that $\U \subset \mathcal{M}$: a universe always consists of ordinary games, not augmented ones.\footnote{Some results can be generalized without much difficulty to ``universes'' of augmented games, with a slightly modified definition of \U-strong. Those generalizations will not be needed here, so we make the assumption $\U \subset \mathcal{M}$ in order to keep the exposition clear.} Thus while $G$ and $H$ may be augmented games, the summand $X$ always ranges over true games.

Note that if $G$ has $\Sigma^L$, then $G$ is automatically Left \U-strong: for if $X$ is a Left end, then $G + X$ also has $\Sigma^L$, so $o(G+X) \geq \N$. This yields the following generalization of Theorem~\ref{theorem:lnsequiv}.

\begin{theorem}[cf.~Theorem~\ref{theorem:lnsequiv}]
\label{theorem:lnsequiv-withsigma}
Let $\U \subset \mathcal{M}$ be a universe and let $G,H \in \Mhat$. Then $G \geq_\U H$ if and only if:
\begin{enumerate}
\item[(a)] For every ordinary $G^R$, either there is some ordinary $H^R$ with $G^R \geq_\U H^R$, or else there is some ordinary $G^{RL}$ with $G^{RL} \geq_\U H$;
\item[(b)] For every ordinary $H^L$, either there is some ordinary $G^L$ with $G^L \geq_\U H^L$, or else there is some ordinary $H^{LR}$ with $G \geq_\U H^{LR}$;
\item[(c)] If $H$ is Left end-like, then $G$ is Left \U-strong;
\item[(d)] If $G$ is Right end-like, then $H$ is Right \U-strong.
\end{enumerate}
\end{theorem}

Note that tombstone options $\Sigma^L$ and $\Sigma^R$ are relevant only to the proviso.

The proof of Theorem~\ref{theorem:lnsequiv-withsigma} will seem achingly familiar. It belongs to a long line of generalizations that began with Conway's work on impartial games (cf.~\cite{Con01}~\cite{Sie13}~\cite{LNS}), and although it is now several steps removed, the contours of the proof still essentially follow \cite[Theorems 75 and~76]{Con01}. Nonetheless, due to the novel setting, we give the proof here in full.

The adjoint (cf.~Definition~\ref{definition:adjoint}) generalizes in the natural way to elements of~\Mhat.

\begin{definition}
\label{definition:augmented-adjoint}
Let $G \in \Mhat$. The \textbf{adjoint} of~$G$, denoted by~$G^\circ$, is defined by
\[
G^\circ = \begin{cases}
\mathspoofwidth{\cgtwo{(G^R)^\circ}{(\fatGL)^\circ}}{\cgstar} & \text{if $G$ has no ordinary options;} \vspace{0.05in} \\
\cgtwo{(G^R)^\circ}{\mathrlap{0}\phantom{(\fatGL)^\circ}} & \parbox{2.4in}{if $G$ has ordinary Right options,\\${}\qquad$ but no ordinary Left options;} \vspace{0.1in} \\
\cgtwo{\phantom{(G^R)^\circ}\mathllap{0}}{(\fatGL)^\circ} & \parbox{2.4in}{if $G$ has ordinary Left options,\\${}\qquad$ but no ordinary Right options;} \vspace{0.05in} \\
\cgtwo{(G^R)^\circ}{(\fatGL)^\circ} & \text{otherwise}.
\end{cases}
\]
\end{definition}

\begin{proposition}[cf.~Proposition~\ref{prop:adjoint-p}]
$G+G^\circ$ is a \P-position, for all $G \in \Mhat$.
\end{proposition}

\begin{proof}
It suffices (by symmetry) to show that $o(G+G^\circ) \geq \P$. Now $G^\circ$ is not Left end-like, so neither is $G+G^\circ$. Moreover, every Right option of $G+G^\circ$ reverts to a position of the form $G^L+(G^L)^\circ$ or $G^R+(G^R)^\circ$, except for Right's move to $G+0$ when $G$ has no ordinary Left options. But in that case, $o(G) \geq \N$ \emph{a priori}.
\end{proof}

\begin{lemma}[cf.~\cite{Sie13}, Theorem~V.6.5]
\label{lemma:notgeq}
Let $G,H \in \Mhat$. If $G \not\geq_\U H$, then:
\begin{enumerate}
\item[(a)] There is some $T \in \U$ such that $o(G+T) \leq \P$ and $o(H+T) \geq \N$.
\item[(b)] There is some $U \in \U$ such that $o(G+U) \leq \N$ and $o(H+U) \geq \P$.
\end{enumerate}
\end{lemma}

\begin{proof}
Since $G \not\geq_\U H$, at least one of (a) or (b) must be true. By symmetry, it suffices to assume (a) and prove~(b).

Fix $T$ as in (a), and put
\[
U = \combgame{\{0,(H^R)^\circ|T\}}.
\]
Now $T \in \U$ and each $(H^R)^\circ$ is a dicot, so since $\U$ is dicotically closed, we have also $U \in \U$. Since $o(G+T) \leq \P$ and $T$ is a Right option of~$U$, we certainly have $o(G+U) \leq \N$. It remains to be shown that $o(H+U) \geq \P$.

But $U$ is not Right end-like, so neither is $H+U$. Moreover, Right's options from $H+U$ all have the form $H^R+U$ or $H+T$. But $o(H+T) \geq \N$ by hypothesis, and each $H^R+U$ has a reverting move to $H^R+(H^R)^\circ$, a $\mathscr{P}$-position.
\end{proof}

\begin{definition}
Let $G,H \in \Mhat$. We say that $G$ is \textbf{\U-downlinked} to~$H$ (by~$T$) if
\[
o(G + T) \leq \mathscr{P} \quad\text{and}\quad o(H + T) \geq \mathscr{P} \qquad \text{for some } T \in \U.
\]
\end{definition}

\begin{lemma}
\label{lemma:downlinked}
Let $G,H \in \Mhat$. Then the following are equivalent:
\begin{enumerate}
\item[(i)] $G$ is \U-downlinked to $H$;
\item[(ii)] No ordinary $G^L \geq_\U H$, and $G \geq_\U$ no ordinary~$H^R$.
\end{enumerate}
\end{lemma}

\begin{proof}
(i) $\Rightarrow$ (ii): Suppose that $G$ is \U-downlinked to~$H$, and fix $T \in \U$ with
\[
o(G + T) \leq \mathscr{P} \quad\text{and}\quad o(H + T) \geq \mathscr{P}.
\]
Then for each ordinary $G^L$, we have $o(G^L + T) \leq \mathscr{N}$, so that
\[o(G^L + T) \not\geq o(H + T).\]
This shows that $G^L \not\geq_\U H$, and a similar argument works for each~$H^R$.

\vspace{0.1in}\noindent
(ii) $\Rightarrow$ (i): Suppose that no ordinary $G^L \geq_\U H$ and $G \geq_\U$ no ordinary~$H^R$. Then for each ordinary Left option $G^L_i$ of~$G$, Lemma~\ref{lemma:notgeq} yields an $X_i \in \U$ such that
\[
o(G^L_i + X_i) \leq \P \quad\text{and}\quad o(H+X_i) \geq \N.
\]
Likewise, for each ordinary Right option $H^R_j$ of~$H$, there is a $Y_j \in \U$ such that
\[
o(G + Y_j) \leq \N \quad\text{and}\quad o(H^R_j + Y_j) \geq \P.
\]
Put
\[
T = \begin{cases}
\cgstar & \text{if neither $G$ nor $H$ has any ordinary options}; \vspace{0.05in} \\
\combgame{\{0|(H^L)^\circ\}} & \parbox{2.4in}{if $G$ has no ordinary options and\\${}\qquad H$ has no ordinary Right options;} \vspace{0.1in} \\
\combgame{\{(G^R)^\circ|0\}} & \parbox{2.4in}{if $H$ has no ordinary options and\\${}\qquad G$ has no ordinary Left options;} \vspace{0.05in} \\
\combgame{\{Y_j,(G^R)^\circ|X_i,(H^L)^\circ\}} & \text{otherwise}.
\end{cases}
\]
In all cases we have $T \in \U$. We claim that $T$ downlinks $G$ to~$H$. We will show that $o(G+T) \leq \P$; an identical argument shows that $o(H+T) \geq \P$.

The definition of $T$ ensures that $T$ is not Left end-like, so neither is $G+T$. To complete the proof, it therefore suffices to show that every ordinary Left option of $G+T$ is losing. But if Left moves to $G_i^L + T$, Right can counter to $G_i^L + X_i$; while if Left moves to $G + (G^R)^\circ$, Right can respond to $G^R + (G^R)^\circ$. Left's moves to $G + Y_j$ lose automatically, by definition of~$Y_j$. The only remaining possibility is Left's move to $G + 0$ in the case where $G$ has no ordinary Right options; but then $G$ is Right end-like, so that $o(G) \geq \N$ \emph{a priori}.
\end{proof}
%

We are now in a position to complete the proof of Theorem~\ref{theorem:lnsequiv-withsigma}.

\begin{proof}[Proof of Theorem~\ref{theorem:lnsequiv-withsigma}]
We prove the Theorem by induction on $G$ and~$H$.

\vspace{0.1in}\noindent
For $\Rightarrow$ (a), fix an ordinary~$G^R$, and suppose (for contradiction) that $G^R \geq_\U$ no ordinary $H^R$ and no ordinary $G^{RL} \geq_\U H$. Then by Lemma~\ref{lemma:downlinked}, we have that $G^R$ is \U-downlinked to~$H$, so that
\[
o(G^R + T) \leq \mathscr{P} \quad\text{and}\quad o(H + T) \geq \mathscr{P}
\]
for some $T \in \U$. It follows that $o(G + T) \leq \mathscr{N}$, contradicting the assumption $G \geq_\U H$.

\vspace{0.05in}\noindent
For $\Rightarrow$ (c), note that if $H$ is Left end-like, then it is automatically Left \U-strong; therefore so is~$G$. Now (b) and (d) follow by symmetry, so this completes the proof of~$\Rightarrow$.

\vspace{0.1in}\noindent
For the $\Leftarrow$ direction, assume that $G$ and~$H$ satisfy (a)--(d); we must show that $o(G+X) \geq o(H+X)$ for each $X \in \U$. We proceed by induction on~$X$. At each step of the induction, it suffices to show that $o^L(G+X) \geq o^L(H+X)$, as the proof that $o^R(G+X) \geq o^R(H+X)$ is identical.

Suppose that $o^L(H+X) = \sL$; we must show that also $o^L(G+X) = \sL$. There are three cases:

\vspace{0.05in}\noindent\emph{Case 1}: $o^R(H+X^L) = \sL$ for some~$X^L$. Then also $o^R(G+X^L) = \sL$, by induction on~$X$, so that $o^L(G+X) = \sL$.

\vspace{0.05in}\noindent\emph{Case 2}: $o^R(H^L+X) = \sL$ for some ordinary~$H^L$. Then the assumption (b) implies either $G^L \geq_\U H^L$ for some ordinary~$G^L$, or else $G \geq_\U H^{LR}$ for some ordinary~$H^{LR}$. In the former case, we have immediately $o^R(G^L+X) = \sL$; while in the latter, we must have $o^L(H^{LR} + X) = \sL$ (since $o^R(H^L+X) = \sL$), so we obtain directly $o^L(G + X) = \sL$.

\vspace{0.05in}\noindent\emph{Case 3}: $H+X$ is Left end-like. Then $X$ must be a Left end and $H$ is Left end-like. It follows from (c) that $G$ is Left \U-strong, whereupon $o^L(G+X) = \sL$, since $X \in \U$.
\end{proof}

\subsection*{Simplification in \Mhat}

\begin{definition}
Let $G \in \Mhat$.
\begin{enumerate}
\item[(a)] An ordinary Left option $G^{L_1}$ is \textbf{\U-dominated (by $G^{L_2}$)} if $G^{L_1} \leq_\U G^{L_2}$ for some other ordinary Left option~$G^{L_2}$.
\item[(b)] An ordinary Left option $G^{L_1}$ is \textbf{\U-reversible (through $G^{L_1R_1}$)} if \\${G \geq_\U G^{L_1R_1}}$ for some ordinary Right option $G^{L_1R_1}$.
\end{enumerate}
\U-dominated and \U-reversible Right options are defined analogously.
\end{definition}

\begin{lemma}
\label{lemma:u-dominated}
Suppose that $G^{L_1}$ is \U-dominated by $G^{L_2}$, and let $G'$ be the game obtained by eliminating $G^{L_1}$ from~$G$. Then $G' \equiv_\U G$.
\end{lemma}

\begin{proof}
We show that $G \geq_\U G'$ and $G' \geq_\U G$ by verifying (a)--(d) in Theorem~\ref{theorem:lnsequiv-withsigma}. For (a) and (d) there is nothing to do, since $G$ and $G'$ have identical Right options. For (c), note that neither $G$ nor $G'$ is a Left end, and $G$ has $\Sigma^L$ if and only if $G'$ does.

Finally, for~(b), note that the ordinary Left options of $G$ pair with those of~$G'$, except for $G^{L_1}$ in the $G' \geq_\U G$ case. But then the assumption $G^{L_2} \geq_\U G^{L_1}$ suffices to verify the maintenance property.
\end{proof}

\begin{lemma}[Ordinary Reversibility]
\label{lemma:reversible1}
Suppose that $G^{L_1}$ is \U-reversible through $G^{L_1R_1}$, and assume that $G^{L_1R_1}$ is not a Left end. Put
\[
G' = \combgame{\{G^{L_1R_1L},G^{L'}|G^R\}},
\]
where $G^{L'}$ ranges over all (ordinary and/or tombstone) Left options of $G$ \emph{except} for~$G^{L_1}$, and $G^{L_1R_1L}$ ranges over all (ordinary and/or tombstone) Left options of $G^{L_1R_1}$. Then ${G' \equiv_\U G}$.
\end{lemma}

\begin{proof}
We show that $G \geq_\U G'$ and $G' \geq_\U G$ by verifying (a)--(d) in Theorem~\ref{theorem:lnsequiv-withsigma}. As in Lemma~\ref{lemma:u-dominated}, for (a) and (d) there is nothing to do, since $G$ and $G'$ have identical Right options.

\vspace{0.1in}\noindent
(i) First we show that $G \geq_\U G'$. To verify (b), note that every Left option of $G'$ of the form $G^{L'}$ is paired by an identical Left option in~$G$, so it suffices to consider an ordinary Left option of $G'$ of the form $G^{L_1R_1L}$. We must show that either some $G^L \geq_\U G^{L_1R_1L}$, or else $G \geq_\U$ some $G^{L_1R_1LR}$; but one or the other must follow from the maintenance property applied to $G \geq_\U G^{L_1R_1}$.

To verify (c): $G'$ cannot be a Left end (by the assumption on $G^{L_1R_1}$). If $G'$ has~$\Sigma^L$, then either $G$ or $G^{L_1R_1}$ must also have~$\Sigma^L$. But since $G \geq_\U G^{L_1R_1}$, either case implies that $G$ is Left \U-strong.

\vspace{0.1in}\noindent
(ii) Next we show that $G' \geq_\U G$. Here (c) is trivial: $G$~cannot be a Left end, and if $G$ has~$\Sigma^L$, then so must~$G'$. For~(b), note that every Left option of $G$ is paired by an identical option of~$G'$, except for~$G^{L_1}$. Thus to complete the proof, it suffices to show that $G' \geq_\U G^{L_1R_1}$.

Now each Left option of $G^{L_1R_1}$ is paired in~$G'$, and each Right option of $G'$ is also an option of~$G$. Since $G \geq_\U G^{L_1R_1}$, this verifies the maintenance property. For the proviso, note that neither $G'$ nor $G^{L_1R_1}$ is a Left end, and if $G^{L_1R_1}$ has~$\Sigma^L$, then so does~$G'$. Finally, if $G'$ is Right end-like, then so is~$G$, and it follows from $G \geq_\U G^{L_1R_1}$ that $G^{L_1R_1}$ is Right \U-strong.
\end{proof}

\begin{lemma}[Atomic Reversibility]
\label{lemma:reversible2}
Suppose that $G^{L_1}$ is \U-reversible through $G^{L_1R_1}$, and assume that $G^{L_1R_1}$ is a Left end. Put
\[
G' = \combgame{\{\Sigma^L,G^{L'}|G^R\}},
\]
where $G^{L'}$ is understood to range over all Left options \emph{except}~$G^{L_1}$. Then ${G' \equiv_\U G}$.
\end{lemma}

\begin{proof}
(i) We first show that $G \geq_\U G'$. Every Right option of $G$ is paired by an identical option of~$G'$, and likewise for every Left option of $G'$ except for~$\Sigma^L$. This verifies the maintenance property. For the proviso, note that since $G \geq_\U G^{L_1R_1}$ and $G^{L_1R_1}$ is a Left end, it follows immediately that $G$ is Left \U-strong.

\vspace{0.1in}\noindent
(ii) Next we show that $G' \geq_\U G$. Every Right option of $G'$ is paired by an identical option of~$G$, and likewise for every Left option of $G$ except for~$G^{L_1}$. Therefore it suffices to show that $G' \geq_\U G^{L_1R_1}$. Now $G^{L_1R_1}$ has no Left option, so consider a Right option $G^R$ of~$G'$. Since $G \geq_\U G^{L_1R_1}$, it must be the case that either some $G^R \geq_\U G^{L_1R_1R}$ or $G^{RL} \geq_\U$ some $G^{L_1R_1}$; either one suffices to establish $G' \geq_\U G^{L_1R_1}$ as well.

For the proviso, note that $G'$ is necessarily \U-strong, since it has $\Sigma^L$ as a Left option.
\end{proof}

\begin{definition}
Let $G \in \Mhat$ and suppose that $G$ has~$\Sigma^L$. The \textbf{Left $\Sigma$-erasure} $\sigma^L(G)$ is the game obtained from $G$ by removing $\Sigma^L$ as a Left option.

The Right $\Sigma$-erasure $\sigma^R(G)$ is defined analogously.
\end{definition}

\begin{lemma}[$\Sigma$-Erasure]
Let $G \in \Mhat$. If $\sigma^L(G)$ is Left \U-strong, then $\sigma^L(G) \equiv_\U G$; likewise for $\sigma^R(G)$.
\end{lemma}

\begin{proof}
It suffices to show $\sigma^L(G) \geq_\U G$ and $G \geq_\U \sigma^L(G)$ using Theorem~\ref{theorem:lnsequiv-withsigma}.

But $G$ and $\sigma^L(G)$ have identical ordinary options, which implies the maintenance property (since every ordinary option of $G$ is paired in~$\sigma^L(G)$, and vice versa). Moreover, $G$~and $\sigma^L(G)$ are both \U-strong, so the proviso is satisfied vacuously.
\end{proof}

\begin{definition}
We say $G$ is \textbf{Left $\Sigma$-erasable (mod~\U)} if $G$ has~$\Sigma^L$, and $\sigma^L(G)$ is Left \U-strong. Right $\Sigma$-erasable is defined analogously; and we say that $G$ is \textbf{$\Sigma$-erasable (mod~\U)} if it is either Left or Right $\Sigma$-erasable.
\end{definition}

\begin{theorem}[Simplest Form Theorem]
Let $G,H \in \Mhat$.
Assume that neither $G$ nor $H$ has any \U-dominated or \U-reversible options, and furthermore that neither $G$ nor $H$ is $\Sigma$-erasable (mod~\U). If $G \equiv_\U H$, then $G \cong H$.
\end{theorem}

\begin{proof}
Let $G^L$ be any ordinary Left option of $G$. It cannot be the case that $G^{LR} \geq_\U H$ for some $G^{LR}$, since we assumed $G$ has no reversible options, so by the maintenance property, there is some ordinary $H^L$ with $G^L \geq_\U H^L$. The same argument shows that there is some ordinary $G^{L'}$ with $H^L \geq_\U G^{L'}$. But $G$ has no dominated options, so in fact $G^L \equiv_\U H^L \equiv_\U G^{L'}$. By induction, $G^L \cong H^L$.

The same argument shows that every ordinary $H^L$ is also a Left option of~$G^L$, and likewise for Right options. This shows that $G$ and $H$ have identical ordinary options.

Now if $H$ has $\Sigma^L$, then it must be \U-strong, so $G$ must also be \U-strong (since $G \geq_\U H$). Since $H$ is not $\Sigma$-erasable and the ordinary options of $G$ are identical to those of~$H$, it follows that $G$ must have $\Sigma^L$ as well. The same argument works for~$\Sigma^R$.
\end{proof}

\begin{definition}
Let $G$ be a partizan game. The \textbf{\U-simplest form} of $G$ is the unique augmented game $K \equiv_\U G$ that has no \U-dominated or \U-reversible options and is not $\Sigma$-erasable (mod~\U).
\end{definition}

We have shown that \U-simplest forms always exist, and they can be calculated (in any universe where $\geq_\U$ is computable) in four steps:
\begin{enumerate}
\item[(i)] Iteratively \emph{bypass} all non-atomic \U-reversible options;
\item[(ii)] \emph{Eliminate} all \U-dominated options;
\item[(iii)] \emph{Replace} all atomic \U-reversible options by $\Sigma^L$ or~$\Sigma^R$;
\item[(iv)] \emph{Erase} $\Sigma^L$ and/or $\Sigma^R$, if erasable.
\end{enumerate}

\begin{example}
Recall the example at the beginning of this section:
\[
A = 0\bar{1}\bar{4}\bar{5}; \qquad B = 0\bar{2}\bar{3}\bar{5}; \qquad \mathcal{U} = \D(A,B),
\]
and
\[
G = \combgame{\{0,\{0,\bar{1}|A\},\{0,\bar{1}|B\}||0,\bar{5}\}}.
\]
The simplest form of $G$ (mod~\U) is given by
\[
K = \combgame{\{0,\Sigma^L|0,\bar{5}\}}.
\]
Note that $K$ would no longer be \U-strong if the $\Sigma^L$ were erased.
\end{example}

\begin{example}
It is important to understand that in addition to the appearance of $\Sigma^L$ and/or $\Sigma^R$ in general simplest forms, it is also possible that new ends might arise that were not present in the original universe. For a simple example, consider
\[
G = \cgtwo{\bar{1}}{\cgstar}
\]
in the dead-ending universe \E. It is easily seen that $G \geq_\E 0$ and, therefore, Left's option $G^L = \bar{1}$ is atomic \E-reversible through~$0$. Thus by $\Sigma$-replacement,
\[
G = \cgtwo{\Sigma^L}{\cgstar}.
\]
Now the $\Sigma^L$ is erasable, so that in fact the simplest form of $G$ is given by
\[
G = \cgtwo{\cdot}{\cgstar},
\]
which is not a dead-end. We will now show that the inclusion of such games poses no difficulties: calculations involving them respect the structure of the original universe.
\end{example}

\subsection*{Congruences in $\Mhat$}

The preceding analysis shows that every $G \in \U$ admits a unique simplest form (mod~\U) in $\Mhat$. In order for the construction to be useful for anything, however, we must show that the games so obtained obey the \textbf{congruence relation}:
\[
\text{If } G \equiv_\U H, \text{ then } G + J \equiv_\U H + J.
\]
When $J \in \U$, then the congruence relation follows immediately from the definition of~$\equiv_\U$. However, it is \emph{not} in general true for all $J \in \Mhat$. In this section, we will prove that it is true for any $J \in \Mhat$ that is obtained by simplifying some $U \in \U$ (and in particular for the simplest form of~$U$).

The fundamental idea is to ``put back'' atomic reversible options in place of $\Sigma^L$ and~$\Sigma^R$ and appeal to the known equivalences on~\U.

\begin{definition}
\label{definition:u-expansion}
Let $G \in \Mhat$ and $U \in \U$. We say that $U$ is a \textbf{\U-expansion} of~$G$ provided that:
\begin{enumerate}
\item[(i)] For every ordinary $G^L$, some $U^L$ is a \U-expansion of~$G^L$;
\item[(ii)] If $G$ is not Left end-like, then every $U^L$ is a \U-expansion of some~$G^L$, as in~(i);
\item[(iii)] If $G$ is Left end-like, then every $U^L$ is either a \U-expansion of some~$G^L$, or else is atomic \U-reversible; and if $U$ is not a Left end, then at least one~$U^L$ is atomic \U-reversible.
\end{enumerate}
\end{definition}

\begin{lemma}
Let $G \in \Mhat$ and suppose that $G$ is the simplest form of some $U \in \U$. Then $G$ has a \U-expansion.
\end{lemma}

\begin{proof}
We can assume that $U$ has no \U-dominated options, since eliminating \U-dominated options does not change the simplest form of~$U$. Likewise, we can assume that $U$ has no ordinary \U-reversible options. Now let $V$ be obtained from $U$ by removing any atomic \U-reversible options that can be removed without (i)~changing the value of~$U$, or (ii)~turning $U$ into an end. We claim that $V$ is a \U-expansion of~$G$.

Now $G$ is obtained from $V$ by $\Sigma$-replacement and $\Sigma$-erasure (applied to various subpositions of~$V$). Thus every ordinary $G^L$ is obtained from some $V^L$ in this way. By induction, $V^L$~is a \U-expansion of~$G^L$, which proves conditions (i) and (ii) in Definition~\ref{definition:u-expansion}.

For (iii), suppose that $G$ is Left end-like. If $G$ has no ordinary Left options, then every $U^L$ must be atomic \U-reversible; otherwise, $G$~has $\Sigma^L$, and so must be obtained by $\Sigma$-replacement of at least one atomic \U-reversible~$U^L$.
\end{proof}

\begin{theorem}
\label{theorem:u-expansion-sums}
Let $G,H \in \hat\U$ and $U,V \in \U$.
If $U$ is a \U-expansion of $G$ and $V$ is a \U-expansion of $H$, then $U + V$ is equal (mod~\U) to a \U-expansion of $G + H$.
\end{theorem}

\begin{proof}
There are four cases. In each case, we show that it is possible to apply a series of reductions to the Left options of $U + V$, resulting in an expression that meets the Left conditions of Definition~\ref{definition:u-expansion}. The analogous Right conditions then follow by symmetry.

\vspace{0.05in}\noindent\emph{Case 1}: Neither $G$ nor $H$ is Left end-like. Then for every ordinary~$G^L$, we know that some $U^L$ is a \U-expansion of~$G^L$, so by induction, $U^L + V$ is equal to a \U-expansion of $G^L + H$. Similarly, for every ordinary $H^L$, some $U + V^L$ is equal to a \U-expansion of $G + H^L$. By the definition of \U-expansion, this must cover all the Left options of $U$ and~$V$, and the conclusion follows by Replacement.

\vspace{0.05in}\noindent\emph{Case 2}: $G$ is Left end-like and $H$ is not. Consider any atomic \U-reversible Left option~$U^{L_1}$, reversible through $U^{L_1R_1} \leq_\U U$. Then $U^{L_1} + V$ is reversible through $U^{L_1R_1} + V$, and since $V$ is not a Left end, we can bypass it. Therefore
\[
U + V \equiv_\U \combgame{\{U^{L_1R_1} + V^L,\ U^{\smash{L'}} + V,\ U + V^L | (U+V)^\mathcal{R}\}},
\]
where $U^{L'}$ is understood to range over all Left options of $U$ except~$U^{L_1}$.

But since $U^{L_1R_1} \leq_\U U$, all the options of the form $U^{L_1R_1} + V^L$ are dominated and can be excluded. We can repeat the procedure with any atomic \U-reversible Left option of~\U, so that in fact
\[
U + V \equiv_\U \combgame{\{U^{\smash{L'}} + V,\ U + V^L | (U+V)^\mathcal{R}\}},
\]
where $U^{L'}$ ranges over all Left options of $U$ that are not atomic \U-reversible.
Since $G + H$ is not Left end-like, the rest of the argument proceeds just as in Case~1.

\vspace{0.05in}\noindent\emph{Case 3}: $H$ is Left end-like and $G$ is not. This case is identical to Case~2, by symmetry.

\vspace{0.05in}\noindent\emph{Case 4a}: $G$ and $H$ are both Left end-like, and neither $U$ nor $V$ is a Left end. Consider a $U^{L_1}$ that is atomic \U-reversible through~$U^{L_1R_1}$. It can be bypassed, so that
\[
U + V \equiv_\U \combgame{\{U^{L_1R_1} + V^L,\ U^{\smash{L'}} + V,\ U + V^L | (U+V)^\mathcal{R}\}},
\]
exactly as in Case~2. Also as in Case~2, options of the form $U^{L_1R_1} + V^L$ are dominated and can be excluded, eventually obtaining
\[
U + V \equiv_\U \combgame{\{U^{\smash{L'}} + V,\ U + V^L | (U+V)^\mathcal{R}\}}.
\]
We can repeat with any $V^{L_1}$ that is atomic \U-reversible:
\[
U + V \equiv_\U \combgame{\{U^L + V^{L_1R_1},\ U^{\smash{L'}} + V,\ U + V^{\smash{L'}} | (U+V)^\mathcal{R}\}}.
\]
Options of the form $U^L + V^{L_1R_1}$ are dominated and can be excluded---\emph{except} in the specific case where $U^L$ is atomic \U-reversible (in which case $U^L + V$ has already been eliminated). But in that case, since $V^{L_1R_1}$ is a Left end, it follows that $U^L + V^{L_1R_1}$ is also atomic \U-reversible. We are left with:
\[
U + V \equiv_\U \combgame{\{T,\ U^{\smash{L'}\vphantom{L}} + V,\ U + V^{\smash{L'}} | (U+V)^\mathcal{R}\}},
\]
where $U^{L'}$ ranges over Left options of $U$ that are not atomic \U-reversible; $V^{L'}$ ranges over Left options of $V$ that are not atomic \U-reversible; and $T$ ranges over various (and at least one) atomic \U-reversible options of $U+V$. Since $G+H$ is Left end-like, this completes the proof.

\vspace{0.05in}\noindent\emph{Case 4b}: $G$ and $H$ are both Left end-like, and at least one of $U$ or $V$ is a Left end. Then for every any atomic \U-reversible $U^L$, the options $U^L + V$ of $U + V$ are already atomic \U-reversible; likewise for any such~$V^L$. It follows that $U + V$ is already a \U-expansion of $G + H$.
\end{proof}

\begin{theorem}
\label{theorem:addition-rule}
If $G,H,J \in \hat\U$ and $G \equiv_\U H$, then $G + J \equiv_\U H + J$.
\end{theorem}

\begin{proof}
Let $U$, $V$, and $W$ be \U-expansions of $G$, $H$, and $J$, respectively. Then by Theorem~\ref{theorem:u-expansion-sums}, $U + W$ and $V + W$ are equal to \U-expansions of $G + J$ and $H + J$, respectively. Therefore
\[
U \equiv G \equiv H \equiv V \pmod\U,
\]
and
\[
G + J \equiv U + W \quad\text{and}\quad H + J \equiv V + W \pmod\U.
\]
But since $W \in \U$, we have that $U + W \equiv_\U V + W$, completing the theorem.
\end{proof}

%
%

\section{Mis\`ere \boldsc{Domineering}}
\label{section:domineering}
\suppressfloats[t]

As an application of the preceding theory, we will give a complete description of the universal closure of \textsc{Domineering}, and we will calculate the simplest forms of $2 \times n$ rectangles for small~$n$.

It is clear that the Left ends of \textsc{Domineering} are sums of $1 \times n$ rectangles, since any contiguous region that is not a $1 \times n$ rectangle must contain two vertically adjacent spaces and, therefore, admit a Left option. In order to isolate the \textsc{Domineering} universe, then, it suffices to characterize the values of $1 \times n$ rectangles. Denote such a rectangle by~$D_n$.


\begin{theorem}
\label{theorem:domineering}
For all $n \geq 0$, the $1 \times n$ \textsc{Domineering} rectangle $D_n$ has the exact value
\begin{center}
\begin{tabular}{l@{\quad if\quad}l}
$k\cdot(\bar{1}0)_\sh$ & $n = \mathspoofwidth{6k + 0}{6k}$ \,or\, $6k + 1$ \bigstrut \\
$k\cdot(\bar{1}0)_\sh + \bar{1}$ & $n = 6k + 2$ \,or\, $6k + 3$ \bigstrut \\
$k\cdot(\bar{1}0)_\sh + \bar{1}0$ & $n = 6k + 4$ \bigstrut \\
$k\cdot(\bar{1}0)_\sh + \bar{2}$ & $n = 6k + 5$ \bigstrut \\
\end{tabular}
\end{center}
\end{theorem}

\noindent Here $k \cdot (\bar{1}0)_\sh$ means the sum of $k$ copies of $(\bar{1}0)_\sh$. Recall that $(\bar{1}0)_\sh = \bar{1}0 + \bar{1}$, so Theorem~\ref{theorem:domineering} implies that the Left ends of \textsc{Domineering} are generated by $\bar{1}$ and $\bar{1}0$ and, therefore, that $\mathcal{D}(\bar{1}0)$ is the universal closure of \textsc{Domineering}.

\begin{proof}
For $n \leq 18$ the theorem is easily verified by direct calculation, so assume $n \geq 19$. It suffices to show that $D_n = D_{n-6} + (\bar{1}0)_\sh$.

Now the options of $D_n$ have the form $D_a + D_b$, with $a + b = n - 2$. Since $n \geq 13$, we may assume without loss of generality that $b \geq 6$. By induction, we may assume that $D_b = D_{b-6} + (\bar{1}0)_\sh$, so that
\[D_a + D_b = D_a + D_{b-6} + (\bar{1}0)_\sh.\]
This shows that every option of $D_n$ is an option of $D_{n-6} + (\bar{1}0)_\sh$. For the converse, the preceding argument shows that every option of the form $D_{n-6}^R + (\bar{1}0)_\sh$ is an option of $D_n$. To see that $D_{n-6} + \bar{1}0$ is also an option of $D_n$, note that either $a \geq 6$ or $b-6 \geq 6$, so one of $D_a$ or $D_{b-6}$ must already have $(\bar{1}0)_\sh$ as a summand.
\end{proof}

It is worth noting that $1 \times n$ \textsc{Domineering} is the one-sided partizan variant of the octal game \textsc{Dawson's Kayles}, and the proof of Theorem~\ref{theorem:domineering} is essentially the same as the proof of the classical Octal Periodicity Theorem~\cite{BCG01}~\cite{GS56}.

\begin{example}
Using \texttt{cgsuite}~\cite{cgsuite}, we can compute the values of $2 \times n$ \textsc{Domineering} rectangles in $\mathcal{D}(\bar{1}0)$. Denote such a rectangle by $\textsc{Dom}(2 \times n)$; then:

\begin{center}
\begin{tabular}{c@{~}c@{~}l}
$\textsc{Dom}(2 \times 1)$ & $=$ & $1$ \vspace{0.1in} \\
$\textsc{Dom}(2 \times 2)$ & $=$ & $\pm1$ \vspace{0.1in} \\
$\textsc{Dom}(2 \times 3)$ & $=$ & $\combgame{\{2,\pm1||\bar{1}|0,\Sigma^R\}}$ \vspace{0.1in} \\
$\textsc{Dom}(2 \times 4)$ & $=$ & $\combgame{\{2,\pm1|0|||\bar{1}|0,\Sigma^R||\cdot\}}$ \vspace{0.1in} \\
$\textsc{Dom}(2 \times 5)$ & $=$ & $\displaystyle
\bigg\{\combgame{
\{3,\{2,\pm1|0\}||0|1\},\pm\kern-2pt\{2,\pm1|0\}|||}$
\\
& &
$\displaystyle\qquad\combgame{\{\{0|1\},\{0|*,\bar{1}0,\pm1\}||\pm(*,1),\{*|*\},\{\bar{1}0|0,\bar{1},\Sigma^R\}\}
}\bigg\}$
\end{tabular}
\end{center}

\noindent Here $\pm1 = \combgame{\{1|\bar{1}\}}$, as per the standard (normal play) convention.

$\textsc{Dom}(2 \times 4)$ is an example of a game whose simplest form is not dead-ending, even though $\mathcal{D}(\bar{1}0)$ is a dead-ending universe.
\end{example}


\section{What's Next?}
\label{section:whats-next}

Taken as a whole, the material in Sections \ref{section:left-ends}--\ref{section:invariant-forms} gives a fairly clear picture of the general subuniverse $\U \subset \E$. So long as $\U$ is computable in the sense of Corollary~\ref{corollary:computable}, we have an algorithm for evaluating $\geq_\U$ and, consequently, an algorithmic construction of \U-simplest forms.

It is natural, then, to wonder about mis\`ere universes not contained in~\E. Intriguingly, the proof of the Simplest Form Theorem, as was already mentioned in the preamble to Section~\ref{section:invariant-forms}, does not require the assumption $\U \subset \E$. Thus simplest forms exist, and can be obtained in essentially the same way, in \emph{any} universe.


However, while simplest forms can be defined and shown to exist in any universe, they are \emph{computable} only when $\geq_\U$ is computable. Since much of the material in Sections \ref{section:left-ends} and~\ref{section:comparison}, and in particular Theorem~\ref{theorem:de-truncation}, requires that $\U \subset \E$, it is unclear how computability of $\geq_\U$ might be established in more general universes. Extending Theorem~\ref{theorem:de-truncation} is therefore an enticing direction for further research.



Santos \cite{San23} has recently introduced a promising family of games known as \textbf{blocking games}, which extend the dead-ending games and enjoy many of their properties.
\begin{definition}[Santos]
A Left end $B$ is \textbf{Left blocked} if, for every~$B^R$, either:
\begin{enumerate}
\item[(i)] $B^R$ is Left blocked, or
\item[(ii)] Some $B^{RL}$ is Left blocked.
\end{enumerate}
A game $G$ is a \textbf{blocking game} if every end in $G$ is blocked. The set of blocking games is denoted by~$\mathcal{B}$.
\end{definition}
Thus if $B$ is Left blocked, then for every move by Right, Left can, at worst, revert it to another Left end in one move.


\subsection*{A Few Examples}

We offer some cursory remarks on more general universes here, leaving a more systematic investigation as an open problem.

The simplest universes not contained in $\E$ are $\D(\combgame{\{\cdot|1\}})$ and $\D(\combgame{\{\cdot|\cgstar\}})$. Intriguingly, $\D(\combgame{\{\cdot|\cgstar\}}) \cong \D$, and $\D(\combgame{\{\cdot|1\}}) \cong \D(\bar{1})$, so these universes introduce no new structure.

\begin{proposition}
\label{proposition:non-de-star}
Let $G = \combgame{\{\cdot|\cgstar\}}$ and let $\U = \D(G)$. Then $G \equiv_\U 0$.
\end{proposition}

\begin{proof}
$G \geq_\U 0$ follows immediately from Theorem~\ref{theorem:lnsequiv} (in fact, $G \geq_\mathcal{M} 0$). For $0 \geq_\U G$, the maintenance property is similarly trivial to verify, so it suffices to show that $G$ is Right \U-strong.

So consider $G + X$, with $X \in \U$ a Right end. Then $X$ necessarily has the form $k \cdot \bar{G}$, for some $k \geq 0$; and from $G + k \cdot \bar{G}$, we claim that Right's move to $\cgstar + k \cdot \bar{G}$ is a winning move.

To see this, note that Right can respond to each of Left's moves on $\bar{G}$ with a reverting move to~$0$; eventually Left will have to move on~$\cgstar$, whence Right has no move.
\end{proof}

\begin{proposition}
\label{proposition:non-de-one}
Let $H = \combgame{\{\cdot|1\}}$ and let $\U = \D(H)$. Then $H \equiv_\U 0$.
\end{proposition}

\begin{proof}
As in Proposition~\ref{proposition:non-de-star}, it suffices to show that $H$ is Right \U-strong, so we consider $H + X$ with $X$ a Right end.

Now $X$ necessarily has the form $k \cdot \bar{H} + n$, for $k \geq 0$ and $n \geq 0$. From
\[
H + k \cdot \bar{H} + n
\]
Right can move to $k \cdot \bar{H} + n + 1$. This is a winning move, since Right can respond to each of Left's moves on $\bar{H}$ with a reverting move to~$0$, and eventually Left will have to move on $n + 1$, whence Right has no move.
\end{proof}

The proofs of Propositions \ref{proposition:non-de-star} and~\ref{proposition:non-de-one} bear a superficial resemblance, but with a subtle difference: \ref{proposition:non-de-star} relies on the fact that $G^R$ is a dicot (and hence cannot show up as a component of a Right end), and \ref{proposition:non-de-one} relies on the fact that $H^R$ is a dead-end (and hence cannot disturb Right's strategy if it does show up).

Remarkably, though, when we put $G$ and $H$ together, the universe strictly expands: $\D(G,H) \not\cong \D(\bar{1})$.

\begin{proposition}
\label{proposition:non-de-one-and-star}
Let $G = \combgame{\{\cdot|\cgstar\}}$, let $H = \combgame{\{\cdot|1\}}$, and let $\U = \D(G,H)$. Then $G \not\equiv_\U 0$.
\end{proposition}

\begin{proof}
$o(1) = \mathscr{R}$, but $o(G + 1) = \P$: on $G + 1$, Right's only move is to $\cgstar + 1$, which Left can revert to~$\cgstar$.
\end{proof}

Here we see a stark difference between blocking and dead-ending universes. Recall that equivalence of dead-ends is absolute: if $G$ and $H$ are dead-ends and $G \equiv_\U H$ in \emph{any} universe, then in fact $G \equiv_\U H$ in \emph{every} universe. The preceding example shows that this is emphatically \emph{not} true for blocked ends.

An interesting direction for further research is to investigate the structure of all universes generated by {birthday-2} ends.

\subsection*{Scoring Universes}

Notably, the construction of \U-simplest forms also generalizes to scoring universes, and to other universes with arbitrary ``atoms'' in the general setting described by Larsson et~al~\cite{LNS}. In particular, if $\U$ is a universe with atoms~$\mathcal{A}$, then in place of $\Sigma^L$ and~$\Sigma^R$ we have a larger family of ``sigmas,''
\[
\Sigma^L_a \quad\text{and}\quad \Sigma^R_a \qquad (a \in \mathcal{A}).
\]
There is an additional reduction of ``removing dominated sigmas,'' in which we remove $\Sigma^L_a$ from $G$ whenever $G$ has another $\Sigma^L_b$ with $b \geq a$. The simplest forms so obtained are unique, provided that $\U$ has \textbf{distinct atoms}: there are no two atoms $a,b \in \A$ with $a \equiv_\U b$. (If the atoms are indistinct, then there is no way for the simplest form to ``know'' which to prefer.)

\subsection*{Is any of this useful for anything?}
\suppressfloats[t]

As we already know from the study of impartial games~\cite{Con01}, the sad reality of the mis\`ere theory is that games admit precious few reductions in practice, and therefore simplest forms tend to be quite complicated. The simplest form of $2 \times 10$ \textsc{Domineering} in $\mathcal{D}(\bar{1}0)$, for example, has $4,569,496$ edges in its game tree, as compared to a mere $33$ in the normal-play simplest form. It is precisely this problem that led to the introduction of mis\`ere quotients for impartial games by Plambeck~\cite{Pla05}~\cite{PS08}.

\begin{figure}
\centering
\begin{tabular}{c|r|r}
\U & \multicolumn{1}{|c|}{$G$} & \multicolumn{1}{|c}{$H$} \\
\hline
$\mathcal{M}$ & 21,946,743 & 2,441,649 \\
\E &  4,590,982 & 2,411,294 \\
$\D(\bar{1}0)$ & 4,569,496 & 2,411,294 \\
$\D(\bar{1})$ & 4,409,745 & 2,411,241 \\
\D & 3,757 & 981,460 \\
Normal Play & 33 & 1,418
\end{tabular}

\vspace{0.15in}
$G = 2 \times 10$ \textsc{Domineering}

\vspace{0.1in}
$H =$ the \textsc{Clobber} position \texttt{"xoxo|oxox|xox."}
\caption{\label{figure:reductions} Size of the simplest forms, modulo various universes, of two example games: the empty $2 \times 10$ \textsc{Domineering} rectangle~$G$, and a \textsc{Clobber} position $H$ on a $3 \times 4$ board. All numbers were calculated using \texttt{cgsuite}.}
\end{figure}

Moreover, the use of fine-grained restricted universes appears to have little practical effect. If we compute the simplest form of $2 \times 10$ \textsc{Domineering} in $\E$ instead of $\mathcal{D}(\bar{1}0)$, for example, we find a game tree with $4,590,982$ edges. Passing from $\E$ to the restricted universe $\mathcal{D}(\bar{1}0)$ therefore eliminates less than $1\%$ of the information content of the simplest form. Experimenting with calculations of this sort, one quickly concludes that in practice, $\mathcal{D}(\bar{1})$ has almost as much resolving power as~$\E$. See Figure~\ref{figure:reductions}.
(The reductions from $\mathcal{M}$ down to~$\E$, and from $\E$ to $\D$, are more significant. Note, however, that the figure for $\D$ in the case of \textsc{Domineering} is somewhat artificial: $G \not\in \D$, so its \D-simplest form will not respect addition.)

The theory of Section~\ref{section:invariant-forms} must therefore be understood as a fascinating theoretical curiosity, one that is unlikely to be applicable to specific case studies in a way that provides much real insight. Nonetheless, it does clarify our structural understanding of games in the general universe. It has long been recognized that atomic reversibility is an impediment to obtaining simplest forms; the results of this paper show that it is in fact the \emph{only} such obstacle.


In hindsight, augmented games can be seen as the most natural way to express simplest forms, even in situations that were previously understood. The dicot introduced at the beginning of Section~\ref{section:invariant-forms},
\[
G = \cgtwo{0,\cgstar}{\cgstar},
\]
is best expressed in simplest form as an augmented game:
\[
G \equiv_\U \cgtwo{0,\Sigma^L}{\cgstar}.
\]
This form is both more concise (fewer edges in its game tree) and more specific: it explicitly acknowledges that the Left option $\cgstar$ does not ``need'' to be a~$\cgstar$; it is there solely to preserve strongness.

\subsection*{What about loopy games?}

Decades ago, Conway introduced a theory of loopy partizan games~\cite{Con78}, beautifully elucidated in \emph{Winning Ways}~\cite[Chapter 11]{BCG01} (see also \cite[Chapter VI]{Sie13}). This theory, elegant as it is, remains curiously incomplete (in normal play---let alone mis\`ere). In particular, the loopy game known as \emph{Bach's Carousel} (after its discoverer, Clive Bach) contains dominated options that cannot be eliminated and reversible ones that cannot be bypassed, and hence it cannot be reduced to simplest form according to the usual methods.

More than 40 years ago, the authors of \emph{Winning Ways} asked:
\begin{quote}
Is there an alternative notion of simplest form that works for \emph{all} finite loopy games (in particular, for the Carousel)?~\cite{BCG01}
\end{quote}
Despite occasional interest over the ensuing decades (Bach's Carousel appeared as Figure 1.1 in this author's Ph.D.~thesis~\cite{Sie05}), this type of position remains poorly understood.

It requires only a smidgen of imagination to envision parallels with the mis\`ere construction of Section~\ref{section:invariant-forms}. ``Reversible options that cannot be bypassed?'' ``Alternative notion of simplest form?'' Perhaps a similar sort of reduction to the one given here, involving the introduction of tombstones or other ``virtual'' representatives of Carousel-type options, might lead to a crisper understanding of the general loopy theory. But Carousels seem a good deal more varied and complicated than atomic reversibility, so novel understandings appear necessary in order to make much progress. Nonetheless, if the partizan mis\`ere theory can be made tractable, surely there is hope that the general loopy theory can, as well.





\nocite{LNNS}
\bibliography{games}

\begin{thebibliography}{10}

\bibitem{BCG01}
E.~R. Berlekamp, J.~H. Conway, and R.~K. Guy.
\newblock {\em Winning Ways for Your Mathematical Plays}.
\newblock A~K Peters, Ltd.~/ CRC Press, Natick, MA, second edition, 2001.

\bibitem{Con78}
J.~H. Conway.
\newblock Loopy games.
\newblock In B.~Bollob\'{a}s, editor, {\em Advances in Graph Theory}, number~3
  in Ann. Discrete Math., pages 55--74, 1978.

\bibitem{Con01}
J.~H. Conway.
\newblock {\em On Numbers and Games}.
\newblock A~K Peters, Ltd.~/ CRC Press, Natick, MA, second edition, 2001.

\bibitem{DRSS15}
P.~Dorbec, G.~Renault, A.~N. Siegel, and E.~Sopena.
\newblock Dicots, and a taxonomic ranking for mis\`ere games.
\newblock {\em J. Combin. Theory Ser.~A}, 130(1):42--63, February 2015.

\bibitem{Ett96}
J.~M. Ettinger.
\newblock {\em Topics in Combinatorial Games}.
\newblock PhD thesis, University of Wisconsin, Madison, 1996.

\bibitem{GS56}
R.~K. Guy and C.~A.~B. Smith.
\newblock The {$G$-values} of various games.
\newblock {\em Proc. Cambridge Philos. Soc.}, 52:514--526, 1956.

\bibitem{LMNRS}
U.~Larsson, R.~Milley, R.~Nowakowski, G.~Renault, and C.~Santos.
\newblock Progress on mis\`ere dead ends: game comparison, canonical form, and
  conjugate inverses.
\newblock preprint.

\bibitem{LNNS}
U.~Larsson, J.~Neto, R.~J. Nowakowski, and C.~P. Santos.
\newblock Guaranteed scoring games.
\newblock {\em Electr. J. Combin.}, 23, 2015.

\bibitem{LNS}
U.~Larsson, R.~J. Nowakowski, and C.~P. Santos.
\newblock Absolute combinatorial game theory.
\newblock preprint.

\bibitem{MO07}
G.~A. Mesdal and P.~Ottaway.
\newblock Simplification of partizan games in mis\`ere play.
\newblock {\em INTEGERS: The Electr. J. Combin. Number Theory}, 7(1):\#G06,
  2007.

\bibitem{Mil13}
R.~Milley.
\newblock {\em Restricted universes of partizan mis\`ere games}.
\newblock PhD thesis, Dalhousie University, 2013.

\bibitem{MR13}
R.~Milley and G.~Renault.
\newblock Dead ends in mis\`ere play: The mis\`ere monoid of canonical numbers.
\newblock {\em INTEGERS: The Electr. J. Combin. Number Theory}, 2013.

\bibitem{Pla05}
T.~E. Plambeck.
\newblock Taming the wild in impartial combinatorial games.
\newblock {\em INTEGERS: The Electr. J. Combin. Number Theory}, 5(1):\#G05,
  2005.

\bibitem{PS08}
T.~E. Plambeck and A.~N. Siegel.
\newblock Mis\`ere quotients for impartial games.
\newblock {\em J. Combin. Theory Ser.~A}, 115(4):593--622, May 2008.

\bibitem{San23}
C.~Santos.
\newblock Personal communication, 2023.

\bibitem{cgsuite}
A.~N. Siegel.
\newblock Combinatorial {Game} {Suite}.
\newblock \raggedright \mbox{\url{http://www.cgsuite.org/}}.

\bibitem{Sie05}
A.~N. Siegel.
\newblock {\em Loopy Games and Computation}.
\newblock PhD thesis, University of California, Berkeley, 2005.

\bibitem{Sie13}
A.~N. Siegel.
\newblock {\em Combinatorial Game Theory}.
\newblock Number 146 in Graduate Studies in Mathematics. American Mathematical
  Society, 2013.

\bibitem{Sie15pmcf}
A.~N. Siegel.
\newblock Partizan mis\`ere canonical form.
\newblock In R.~J. Nowakowski and D.~Wolfe, editors, {\em Games of No
  Chance~4}, MSRI Publications. Cambridge University Press, Cambridge, 2015.

\bibitem{Smi66}
C.~A.~B. Smith.
\newblock Graphs and composite games.
\newblock {\em J. Combin. Theory, Ser.~A}, 1:51--81, 1966.

\end{thebibliography}

\end{document}